\documentclass[11pt,a4paper]{article}

\usepackage{header}
\usepackage{abbrev}

\begin{document}

\title{Shooting methods for computing geodesics on the Stiefel manifold}

\author{Marco Sutti\thanks{Mathematics Division, National Center for Theoretical Sciences, Taipei, Taiwan (\email{msutti@ncts.tw}).}\hspace{2mm}\orcidlink{0000-0002-8410-1372}}

\date{\today}

\maketitle

\begin{abstract}
This paper shows how to use the shooting method, a classical numerical algorithm for solving boundary value problems, to compute the Riemannian distance on the Stiefel manifold $\Stnp$, the set of $ n \times p $ matrices with orthonormal columns. The main feature is that we provide neat, explicit expressions for the Jacobians. To the author's knowledge, this is the first time some explicit formulas are given for the Jacobians involved in the shooting methods to find the distance between two given points on the Stiefel manifold. This allows us to perform a preliminary analysis for the single shooting method. Numerical experiments demonstrate the algorithms in terms of accuracy and performance. Finally, we showcase three example applications in summary statistics, shape analysis, and model order reduction.



\end{abstract}

\section{Introduction}

The object of study in this paper is the compact Stiefel manifold, i.e., the set of $ n \times p $ matrices with orthonormal columns
\[
   \Stnp = \left\lbrace X \in \R^{n \times p}: \ X\tr\! X = I_p \right\rbrace.
\]
There are applications in several areas of mathematics and engineering that deal with data that belong to $\Stnp$. Domains of applications include numerical optimization, imaging, and signal processing. Some applications, like finding the Riemannian center of mass, require evaluating the geodesic distance between two arbitrary points on $\Stnp$. Since no explicit formula is known for computing the distance on $\Stnp$, one has to resort to numerical methods. 

In this paper, we are concerned with computing the Riemannian distance between two given points on the Stiefel manifold. As we shall see, the distance between two points on a manifold is related to the concept of minimizing geodesic\footnote{Geodesics are generally defined as critical points of the length functional, and as such, they may or may not be minima. A minimizing geodesic is one that minimizes the length functional. We introduce the notion of geodesics in section~\ref{sec:geodesic_exp_log}.}.
The problem can be briefly formulated as follows. Given two points $X$, $Y$ on $\Stnp$ that are sufficiently close to each other, finding the distance between them is equivalent to finding the tangent vector $\xi^{\ast} \in \mathrm{T}_{X}\Stnp$ with the shortest possible length such that \protect{\cite{Lee:2018,boumal2023intromanifolds}}
\[
\Exp_{X}(\xi^{\ast}) = Y,
\]
where $\Exp_{X}$ denotes the Riemannian exponential mapping at $X$. 
The solution to this problem is equivalent to the Riemannian logarithm of $Y$ with base point $X$
\[
 \xi^{\ast} = \Log_{X}(Y).
\]
The sought distance between $X$ and $Y$ is then given by the norm of $ \xi^{\ast} $.

Figure~\ref{fig:single_shooting} provides an artistic illustration of the problem. The latter will be stated in more detail in section~\ref{sec:problem_statement}.

\begin{figure}[htbp]
\centering
\includegraphics[width=0.50\columnwidth]{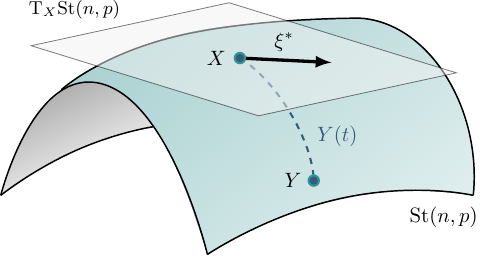}
\caption{Illustration of the problem statement.}
\label{fig:single_shooting}
\end{figure}

It is interesting to note that, for some manifolds, explicit formulas exist for computing the Riemannian distance. This is the case of the Grassmann manifold $ \Grass(n,p) $, which is the set of $ n \times p $ vector subspaces of $ \R^{n} $. For instance, let $ \cX $ and $ \cY $ be two subspaces belonging to $ \Grass(n,p) $, then the distance between $ \cX $ and $ \cY $ is
\[
   d(\cX,\cY) = \sqrt{\theta_{1}^{2} + \cdots + \theta^{2}_{p}},
\]
where $ \theta_{i} $, $ i = 1, \ldots, p $, are the principal angles between $ \cX $ and $ \cY $ (see \protect{\cite[Theorem~8]{Wong:1967}} and \protect{\cite[\S 3.8]{AMS:2004}}).
In contrast, there is no such closed-form solution for the Stiefel manifold. This motivates us to consider numerical methods. In general, the problem of finding the distance given two points on a Riemannian manifold is related to the Riemannian logarithm function (more details later in section~\ref{sec:geodesic_exp_log}). Several authors have already tackled the problem of computing the Riemannian logarithm on the Stiefel manifold. These contributions are detailed in section~\ref{sec:other_approaches}.

\subsection{Contributions}

In this work, we use the shooting methods, which are classical numerical algorithms for solving boundary value problems, to compute the distance on the Stiefel manifold $\Stnp$. These methods enjoy local quadratic convergence properties when close to the solution. 
These methods are not new (thorough coverage of the shooting methods is given, e.g., in \protect{\cite{Stoer:1991}}), neither is their application to computing the Riemannian distance on the Stiefel manifold (see, e.g., \protect{\cite{Bryner:2017}}). Still, to the author's knowledge, this is the first time some neat explicit expressions are given for the Jacobians involved in the shooting methods to compute the distance between two given points on the Stiefel manifold. Hence, there is no need for finite difference approximations. This allows us to conduct a preliminary analysis.
In particular, the main contributions of this paper are as follows.

\begin{enumerate}[label=(\roman*)]
\item We provide shooting methods for computing geodesics on the Stiefel manifold using the canonical metric, with neat formulas for the Jacobians involved in the algorithm.
\item We provide an elegant way to start the algorithm and get the desired quadratic convergence.
\item We analyze the explicit expression of the Jacobian of the matrix exponential involved in the single shooting method to find a connection between numerical linear algebra and differential geometry properties.
\item Numerical experiments demonstrate the algorithms in terms of performance and accuracy.
\item Applications in three different domains are provided to show how the algorithms may be used in practical circumstances.
\end{enumerate}

Part of this work has already appeared in the author's Ph.D. thesis~\protect{\cite{Sutti:2020b}}. Here, it is presented in an independent form; we extend it and add some new applications.

\subsection{Applications and motivation}

Many works have used the Stiefel manifold in their applications. To provide some motivation for the present work, this section summarizes a few applications that explicitly compute the geodesic distance according to the main areas of applications.

In affine invariant shape analysis, \protect{\cite{Younes:2008}} studied a specific metric on plane curves that has the property of being isometric to classical manifolds (like the sphere, complex projective plane, Stiefel and Grassmann manifolds) modulo change of parametrization. Moreover, they provided experimental results that explicitly compute minimizing geodesics between two closed curves.

In the context of shape analysis of closed curves, \protect{\cite{Srivastava:2016}} studied the space of functions representing unit-length, planar, closed curves, which can be shown to be a Stiefel manifold.
\protect{\cite[\S 4.2]{Ring:2012}} provided an application for image segmentation on the Stiefel manifold using a Riemannian variant of the classical BFGS algorithm. This is compared to the work of~\protect{\cite{Sundaramoorthi:2011}}, where the authors used geodesic retractions based on the matrix exponential.
The more general reference \protect{\cite[Chapter~6]{Kendall:1999}} also contains a discussion on the Stiefel manifold and shape spaces. 
\protect{\cite{Bryner:2017}} also proposed some numerical applications on the pre-shape space.
 
\protect{\cite{Cetingul:2009}} investigated the intrinsic mean shift algorithm for clustering on Stiefel and Grassmann manifolds. \protect{\cite{Turaga:2008,Turaga:2011}} investigated applications of the Stiefel manifold in computer vision and pattern recognition to develop accurate inference algorithms. Vision applications such as activity recognition, video-based face recognition, shape classification, and unsupervised clustering were targeted. In particular, step 3 of Algorithm 1 in~\protect{\cite{Turaga:2011}} computes the inverse exponential map, but it was unclear how this was achieved.

The low-rank representation (LRR) is a widely used technique in computer vision and pattern recognition for data clustering models. \protect{\cite{Yin:2015}} extended the LRR from Euclidean space to the manifold-valued data on the Stiefel manifold by incorporating the intrinsic geometry of the manifold. They acknowledged that, in general, it is pretty hard to compute the log mapping for the Stiefel manifold. Consequently,  they used the retraction map (a first-order approximation to the exponential mapping~\protect{\cite{Absil:2012}}) instead of the exponential map because of its reduced computational cost.

More recently, \protect{\cite{LiMa:2022}} proposed a generalization of the federated learning framework to Riemannian manifolds. In particular, they consider the kPCA problem on the Stiefel manifold. Even though they initially discuss the Riemannian logarithm mapping, they finally adopt a retraction in the numerical implementations, similarly to what was done by \protect{\cite{Yin:2015}}.

\subsection{Related works and other approaches} \label{sec:other_approaches}

Shooting methods are not the only option to solve the endpoint geodesic problem; other numerical algorithms have been proposed.

The leapfrog algorithm by Noakes~\protect{\cite{Noakes:1998}} is based on partitioning the original problem into smaller subproblems. This method has global convergence properties, but it slows down for an increasing number of subproblems or when the solution is approached \protect{\cite[\S 1]{Kaya:2008}}. Moreover, Noakes realized that his leapfrog algorithm was in some way imitating the Gauss--Seidel method \protect{\cite[\S 1]{Noakes:1998}}. This connection has been explored in~\protect{\cite{Sutti_V:2023}}.

\protect{\cite{Bryner:2017}} proposed two numerical schemes, shooting method and path-straightening, to compute endpoint geodesics on the Stiefel manifold by considering them as an embedded submanifold of the Euclidean space. From the matrix algebra perspective, Rentmeesters \protect{\cite{Rentmeesters:2013}}. Zimmermann \protect{\cite{Zimmermann:2017,Zimmermann:2019}} derived algorithms for evaluating the Riemannian logarithm map on the Stiefel manifold with respect to the canonical metric, which is locally convergent and depends upon the definition of the matrix logarithm function. \protect{\cite{Zimmermann:2022}} provided a unified method to deal with the endpoint geodesic problem on the Stiefel manifold with respect to a family of metrics.

Recently \protect{\cite{Noakes:2022}} proposed an alternative algorithm to find geodesics joining two given points. Like leapfrog, this method exploits the shooting method to compute geodesics joining junction points.

\subsection{Notation}\label{sec:notation}

We list here the notations and symbols adopted in the paper in order of appearance. Symbols only used in one section and notations specific to the applications of section~\ref{sec:applications} are typically omitted from this list. Some symbols are inevitably overloaded, but their meaning should be clear from the context. 

\begin{table}[htbp]
   \begin{center}
      \begin{tabular}{ll} 
          $ \Stnp $                          &  Stiefel manifold of orthonormal $n$-by-$p$ matrices \\
          $ X $, $ Y $, $ Y_{0} $, $ Y_{1} $ &  Elements of $ \Stnp $ \\
          $ I_{p} $                          &  The identity matrix of size $p$-by-$p$ \\
          $ \mathrm{T}_{X}\Stnp $            &  Tangent space at $ X $ to the Stiefel manifold $ \Stnp $ \\
          $ \xi^{\ast} $                     &  A tangent vector that we want to recover \\
          $ \Exp_{X} $                       &  Riemannian exponential map at $X$ \\
          $ \Log_{X} $                       &  Riemannian logarithm map at $X$ \\
          $ O_{p} $                          &  The null matrix of size $p$-by-$p$ \\
          $ \cS_{\mathrm{sym}}(p) $          &  Space of $ p $-by-$ p $ symmetric matrices \\
          $ \cS^{n-1} $                      &  The unit sphere embedded in $\R^{n}$ \\
          $ \On $                            &  The orthogonal group of $n$-by-$n$ orthogonal matrices \\
          $ X_{\perp} $                      &  An orthonormal matrix whose columns span \\
                                             &  the orthogonal complement of $ \mathrm{span}(X) $ \\
          $ \cS_{\mathrm{skew}}(p) $         &  Space of $ p $-by-$ p $ skew-symmetric matrices \\ 
          $ \Omega $                         &  An element of $ \cS_{\mathrm{skew}}(p) $ \\
          $ K $                              &  A matrix in $ \R^{(n-p) \times p} $ \\
          $ \bar{n} $                              &  The dimension of $ \Stnp $, equal to $ np - \tfrac{1}{2}p(p+1) $ \\
          $ \cM $                            &  Generic manifold \\
          $ \mathrm{T}_{x}\cM $              &  Tangent space at $ x $ to the Stiefel manifold $ \cM $ \\
          $ \langle \cdot, \cdot \rangle_{x} $   &  Inner product on the tangent space $ \mathrm{T}_{x}\cM $ \\
          $ g $                              &  Riemannian metric \\
          $ \gamma(t) $                      &  Parametrized curve on the manifold $\cM$ \\
          $ d $                              &  Riemannian distance function \\
          $ d(x,y) $                         &  Riemannian distance between two points $x$ and $y$ \\
          $ \P_{X} $                         &  The projector onto the tangent space $ \mathrm{T}_{X}\Stnp $ \\
          $ \inj_{X}(\cM) $                  &  Injectivity radius of $\cM$ at $X$ \\
          $ \inj(\cM) $                      &  Global injectivity radius of $\cM$ \\
          $ \| \cdot \|_{\F} $               &  Frobenius norm \\
          $ \| \cdot \|_{\mathrm{c}} $       &  Canonical norm \\
          $ A $   & The matrix $ \begin{bmatrix}   \Omega   &   -K\tr \\   K    &    O_{n-p}  \end{bmatrix} $ \\
          $ Z_{1}(t) $ or $ Y(t) $           &  A geodesic on $\Stnp$ \\
          $ Z_{2}(t) $ or $ \dt{Y}(t) $      &  The derivative of a geodesic \\
          $ F $                              &  The nonlinear function $ Z_{1}(1,\xi) - Y_{1} $ \\
          $ F^{(k)} $                        &  The nonlinear function $F$ evaluated at iteration $ k $ \\
          $ \delta \xi^{(k)} $               &  The residual at iteration $ k $ in the single shooting method \\
          $ J_{Z_{1}}^{x} $                  &  The Jacobian matrix of $Z_{1}$ with respect to $x$ \\
          $ J_{A(x)}^{x} $                   &  The Jacobian matrix of $A$ with respect to $x$ \\
          $ J_{\exp(A)}^{A} $                &  The Jacobian matrix of $\exp(A)$ with respect to its argument $A$ \\
          $\sigma_{\min}$, $\sigma_{\max}$   &  Smallest and largest singular values of a matrix \\
          $ \epsmach $                       &  The machine epsilon, in double precision
      \end{tabular}
   \end{center}
\end{table}

\subsection{Outline of the paper}

The remaining part of this paper is organized as follows.
Section~\ref{sec:stiefel_manifold} introduces the geometry of the Stiefel manifold. The reader who is familiar with Riemannian geometry, particularly the geometry of the Stiefel manifold, might want to skip this section.
Section~\ref{sec:problem_statement} presents the problem statement, which is the focus of this work. Section~\ref{sec:single_shooting} describes the single shooting method, and section~\ref{sec:multiple_shooting} is devoted to the multiple shooting method. Numerical experiments are presented in section~\ref{sec:numerical_experiments}, while section~\ref{sec:applications} focuses on more practical applications. Finally, we conclude the paper by summarizing the contributions and providing future research outlooks in section~\ref{sec:conclusions}.

\section{Geometry of the Stiefel manifold}\label{sec:stiefel_manifold}

The set of all $ n \times p $ orthonormal matrices, i.e.,
\[
   \Stnp = \lbrace X \in \Rnp \colon X\tr \! X = I_{p} \rbrace,
\]
endowed with its submanifold structure, as discussed below, is called an orthogonal or compact Stiefel manifold.
It is a subset of $ \Rnp $, and it can be proven that it has the structure of an embedded submanifold of $ \Rnp $ \protect{\cite[\S 3.3.2]{AMS:2008}}. 

We state some basic properties of the Stiefel manifold $ \Stnp $. It is closed because it is the inverse image of the closed set $ \lbrace O_{p} \rbrace $ under the continuous function $ F \colon \Rnp \to \mathcal{S}_{\mathrm{sym}}(p) $. It is bounded; each column of $ X \in \Stnp $ has norm $ 1 $, so the Frobenius norm of $ X $ equals $\sqrt{p}$. It is compact since it is closed and bounded. This follows from the Heine--Borel theorem \protect{\cite[\S A.2]{AMS:2008}}.

The Stiefel manifold $ \Stnp $ may degenerate to some special cases. For $ p = 1 $, it reduces to the unit sphere $ \cS^{n-1} $ in $ \R^{n} $.
For $ p = n $, the Stiefel manifold becomes the orthogonal group $ \On $, whose dimension is $ \tfrac{1}{2}n(n-1) $.

\subsection{Tangent spaces} \label{sec:tgspaces_stiefel}

The tangent space to a manifold at a given point can be seen as a local vector space approximation to the manifold at that point. In practice, it is useful to perform the operations of vector addition and scalar multiplication, which would otherwise be impossible to perform on a manifold without leaving it, due to the manifold's curvature. Endowed with a Euclidean inner product, this vector space becomes a Euclidean space where we also have a notion of lengths. Here, we will directly focus on the tangent space to the Stiefel manifold. For a more precise definition of a tangent space in the general case, we refer the reader to \protect{\cite{AMS:2008}}.

The \emph{tangent space to the Stiefel manifold at a point} $X$ is \protect{\cite[\S 3.5.2]{AMS:2008}}
\[
   \mathrm{T}_{X}\Stnp = \lbrace Z \in \Rnp \colon  X\tr Z + Z\tr X = 0 \rbrace.
\]
An alternative way to characterize the tangent space $ \mathrm{T}_{X} \Stnp $ is as follows. Let $ X_{\perp} $ be an orthonormal matrix whose columns span the orthogonal complement of $ \mathrm{span}(X) $.
Since both $ X $ and $ X_{\perp} $ are orthonormal, together they form an orthonormal basis of the space $ \Rnp $, so that we can decompose any tangent vector $ \dt{X} $ on this basis as
\[
   \dt{X} = X \Omega + X_{\perp} K,
\]
$ \Omega $ being a $p$-by-$p$ skew-symmetric matrix, $ \Omega \in \cS_{\mathrm{skew}}(p) $, and $ K \in \R^{(n-p) \times p} $, with no restriction on $ K $. So the tangent space to the Stiefel manifold can also be characterized by
\begin{equation}\label{eq:tg_space_stiefel}
   \mathrm{T}_{X}\Stnp = \lbrace X \Omega + X_{\perp} K \colon  \Omega = -\Omega\tr, \ K \in \R^{(n-p)\times p} \rbrace.
\end{equation}
With this characterization in mind, and with the fact that $ \dim\!\big(\Stnp\big) = \dim\!\big(\mathrm{T}_{X}\Stnp\big) $, it is straightforward to work out the \emph{dimension of the Stiefel manifold} as
\[
   \dim(\Stnp) = \dim(\cS_{\mathrm{skew}}(p)) + \dim(\R^{(n-p) \times p}) = \tfrac{1}{2}p(p-1) + (n-p) p = np - \tfrac{1}{2}p(p+1).
\]
For ease of notation, we denote the dimension of $ \Stnp $ by $ \bar{n} $.

\subsection{Riemannian metric and distance}\label{sec:riem_metric_dist}

To define a distance on a given manifold $\cM$, we still need a notion of length that applies to tangent vectors. To this aim, we endow the tangent space $ \mathrm{T}_{x}\cM $ with an \emph{inner product} $ \langle \cdot, \cdot \rangle_{x} $, i.e., a bilinear, symmetric positive definite form. The subscript $ x $ in $ \langle \cdot, \cdot \rangle_{x} $ indicates that, in general, the inner product depends on the point $ x \in \cM $.
The inner product $ \langle \cdot, \cdot \rangle_{x} $ induces a norm $ \| \xi_{x} \|_{x} = \sqrt{\langle \xi_{x}, \xi_{x} \rangle_{x}} $ on $ \mathrm{T}_{x}\cM $.
The introduction of the inner product structure permits to define the notion of Riemannian manifold.
A manifold $ \cM $ endowed with a smoothly-varying inner product (called \emph{Riemannian metric} $ g $) is called \emph{Riemannian manifold}.
Strictly speaking, a Riemannian manifold is a couple $ (\cM,g) $, i.e., a manifold with a Riemannian metric.
A vector space endowed with an inner product structure is a particular case of Riemannian manifold called \emph{Euclidean space}.

The \emph{length of a curve} $ \gamma \colon [a,b] \to \cM $ on a Riemannian manifold $ ( \cM, g ) $ is
\[
   L(\gamma) = \int_{a}^{b} \sqrt{g(\dt{\gamma}(t),\dt{\gamma}(t))} \, \mathrm{d}t.
\]

The \emph{Riemannian distance} is defined as the shortest path between two points $ x $ and $ y $
\[
   d \colon \cM \times \cM \to \R \colon d(x,y) = \inf_{\Gamma} L(\gamma),
\]
where $ \Gamma $ denotes the set of all curves $ \gamma $ in $ \cM $ joining points $ x $ and $ y $.

\subsection{Normal space}\label{sec:normal_space}

Let $ \cM $ be an embedded submanifold of a Riemannian manifold $ \ \ \overbar{\cM} $. Since $ \cM $ is a submanifold, it can inherit the Riemannian metric from its embedding space $ \ \overbar{\cM} $
\[
   g_{x} (\xi,\zeta) = \bar{g}_{x}(\xi,\zeta), \qquad \xi, \zeta \in \mathrm{T}_{x}\cM.
\]

The orthogonal complement of $ \mathrm{T}_{x}\cM $ in $ \mathrm{T}_{x} \ \overbar{\cM} $ is called \emph{normal space} to $ \cM $ at $ x $ and it is defined by
\[
   (\mathrm{T}_{x}\cM)^{\perp} = \left\lbrace \xi \in \mathrm{T}_{x} \ \overbar{\cM} \colon \bar{g}_{x}(\xi,\zeta) = 0, \quad \forall \zeta \in \mathrm{T}_{x}\cM \right\rbrace.
\]

Any tangent vector $ \xi \in \mathrm{T}_{x} \ \overbar{\cM} $ can be uniquely decomposed into
\[
   \xi = \P_{x}\xi + \P_{x}^{\perp} \xi,
\]
where $ \P_{x} $ and $ \P_{x}^{\perp} $ denote the orthogonal projections onto $ \mathrm{T}_{x}\cM $ and $ (\mathrm{T}_{x}\cM)^{\perp} $, respectively.

The tangent space to $ \Stnp $ at $ X $ is given by~\eqref{eq:tg_space_stiefel}. The Riemannian metric inherited by $ \mathrm{T}_{X}\Stnp $ from the embedding space $ \Rnp $ is
\[
   \langle \xi, \eta \rangle_{X} = \trace(\xi\tr \eta).
\]
The normal space is given by those matrices $ M $ such that 
\[
   \langle \xi, M \rangle_{X} = 0, \quad \forall \xi \in \mathrm{T}_{X}\Stnp.
\]
Take $M$ in the form $ M = XS $, with $ X \in \Stnp $ and $ S $ a $p$-by-$p$ symmetric matrix, $ S \in \cS_{\mathrm{sym}}(p) $. Then, one can easily verify that
\[
   \langle \xi, M \rangle_{X} = \trace(\xi\tr M) = \trace((\Omega_{\xi}\tr X\tr + K_{\xi}\tr X_{\perp}\tr)\,XS) = \trace(\Omega_{\xi}\tr S) = 0.
\]
Thus, the\emph{ normal space to the Stiefel manifold at a point} $ X $ is given by
\[
   (\mathrm{T}_{X}\Stnp)^{\perp} = \lbrace XS \colon S \in \cS_{\mathrm{sym}}(p)\rbrace.
\]

\subsection{Projectors} \label{sec:projectors}

The \emph{projection onto the tangent space} $ \mathrm{T}_{X}\Stnp $ is
\[
   \P_{X}\xi = X \mathrm{skew}(X\tr \xi) + (I-XX\tr)\,\xi,
\]
and the \emph{projection onto the normal space} $ (\mathrm{T}_{X}\Stnp)^{\perp} $ is
\[
   \P_{X}^{\perp}\xi = X \mathrm{sym}(X\tr \xi).
\]
Before stating our problem, we still need to introduce the notion of geodesics.

\subsection{Geodesics, exponential mapping and logarithm mapping} \label{sec:geodesic_exp_log}

We first give a general survey about geodesics, then switch to the particular case of the Stiefel manifold in the next section.
Geodesics are defined as curves with zero ``acceleration'', i.e., they solve the second-order ordinary differential equation (ODE)
\[
   \frac{\D^2}{\mathrm{d}t^2} \, \gamma(t) = 0,
\]
where $ \tfrac{\D^2}{\mathrm{d}t^2} $ denotes the \emph{acceleration vector field}.
Geodesics allow us to introduce the \emph{Riemannian exponential} $ \Exp_{x}\colon \mathrm{T}_{x}\cM \to \cM $ that maps a tangent vector $\xi = \dt{\gamma}(0) \in \mathrm{T}_{x}\cM $ to the geodesic endpoint $\gamma(1) = y $: $ \Exp_{x}(\xi) = y$. Figure~\ref{fig:exponential_map} illustrates these concepts for the unit sphere $ \cS^{2} $.

\begin{figure}[htbp]
\centering
\includegraphics[width=0.50\columnwidth]{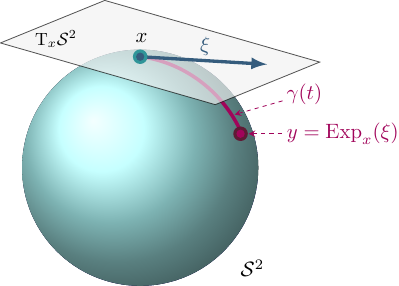}
\caption{The Riemannian exponential map on the sphere.}
\label{fig:exponential_map}
\end{figure}

The Riemannian exponential is a local diffeomorphism, i.e., it is locally invertible, and its inverse is called the \emph{Riemannian logarithm} of $y$ at $x$: $ \Log_{x}(y) = \xi $.

Thanks to a result of Riemannian geometry known as \emph{Gauss's lemma}, the exponential map can be locally understood as a \emph{radial isometry} \protect{\cite[Lemma~3.5]{doCarmo:1992}}. This means that one can measure the distance between two sufficiently close points on the manifold by computing the norm of the corresponding vector in the tangent space, i.e.,
\[
   d(x,y) = \| \Log_{x}(y) \|.
\]
This is the result we anticipated in the introduction, which is very useful in practical applications since it allows us to compute the distance $d(x,y)$ as the norm of the vector $\Log_{x}(y)$. The problem of our interest then becomes how to compute $\Log_{x}(y)$.

The diffeomorphicity of the exponential mapping is closely linked to the behavior of geodesics. While in Euclidean geometry straight lines are also distance-minimizing curves, in Riemannian geometry a geodesic $ \gamma \colon [0,t] \to \cM $ emanating from a point $ x $ is distance-minimizing only for small values of the parameter $t>0$.
In general, there exists a point $\gamma(t_{c})$, called \emph{cut point}, where the distance-minimizing property first breaks down \protect{\cite[Chapter~III]{Sakai:1996}}. The union of the cut points of all geodesics emanating from $ x $ is called \emph{cut locus} of $ x $; it is the boundary of the (star-shaped) domain in which $ \Exp_{x}\colon \mathrm{T}_{x}\cM \to \cM $ is a diffeomorphism. 
The cut locus is closely linked not only to local properties such as the curvature of $\cM$, but also to global topological properties \protect{\cite{Sakai:1996,Afsari:2013}}.

The \emph{injectivity radius} at a point $ x $ of a Riemannian manifold $\cM$ is the largest radius for which the exponential map $\Exp_{x}$ is a diffeomorphism from the tangent space to the manifold; it is the least distance from $ x $ to the cut locus of $ x $. The \emph{global injectivity radius} of a manifold is the infimum of all the injectivity radii at all points of the manifold.
Given two points $ x $ and $ y $ on a manifold $ \cM $, if $ d(x,y) < \inj(\cM) $, then there exists a unique minimizing geodesic from $ x $ to $ y $. For the Stiefel manifold, the injectivity radius is lower bounded by $ 0.89\,\pi $ \protect{\cite[(5.13)]{Rentmeesters:2013}}.

\subsection{Geodesics on the Stiefel manifold}

As mentioned in section~\ref{sec:riem_metric_dist}, a Riemannian metric has to be specified to turn $\Stnp$ into a Riemannian manifold; generally, different choices are possible.
In this paper, we consider the non-Euclidean \emph{canonical metric} inherited by $\Stnp$ from its definition as a quotient space of the orthogonal group \protect{\cite[(2.39)]{Edelman:1998}}. Given $ X \in \Stnp $ and $ \xi, \zeta \in \mathrm{T}_{X}\Stnp $, the canonical metric reads
\begin{equation}\label{eq:formula_canonical_metric}
    g_{c}(\xi,\zeta) = \trace\!\big( \xi\tr ( I - \tfrac{1}{2}XX\tr) \, \zeta \big).
\end{equation}
The canonical metric induces the \emph{canonical norm}, defined as
\[
    \| \xi \|_{\mathrm{c}} = \sqrt{g_{c}(\xi,\xi)}.
\]
The reader can verify that
\[
   \| \xi \|_{\mathrm{c}}^{2} = \tfrac{1}{2} \| \Omega \|_{\F}^{2} + \| K \|_{\F}^{2}.
\]
The \emph{embedded metric} is the metric inherited by the Stiefel manifold as an embedded submanifold of $ \R^{n} $, i.e., $ g_{e}(\xi,\xi) = \trace( \xi\tr \xi ) $, and leads to very similar derivations. With the embedded metric, the induced norm is simply the Frobenius norm
\[
    \| \xi \|_{\mathrm{e}} = \sqrt{g_{e}(\xi,\xi)} = \sqrt{\trace( \xi\tr \xi )} \eqqcolon \| \xi \|_{\F},
\]
and
\[
   \| \xi \|_{\F}^{2} = \| \Omega \|_{\F}^{2} + \| K \|_{\F}^{2}.
\]
The only difference with respect to the canonical norm is the absence of the factor $\tfrac{1}{2}$ in front of $\| \Omega \|_{\F}^{2}$. This calculation highlights the fact that in contrast to the embedded norm, the canonical norm only takes into account once the $ \tfrac{1}{2} p (p-1) $ coefficients of $ \Omega $. Indeed, in the remaining part of this paper, we will only use the canonical metric. 

By endowing the Stiefel manifold with the canonical metric, one can derive the following \emph{second-order ordinary differential equation for the geodesic} \protect{\cite[(2.41)]{Edelman:1998}}
\begin{equation}\label{eq:ode_geodesics}
   \ddt{Y} + \dt{Y}\dt{Y}\tr Y + Y \big( (Y\tr\dt{Y})^{2} + \dt{Y}\tr\dt{Y}\big) = 0,
\end{equation}
where $Y \equiv Y(t)$.

An explicit formula for a geodesic that realizes a tangent vector $ \xi $ with base point $ Y_{0} $ is \protect{\cite[(2.42)]{Edelman:1998}}
\begin{equation}\label{eq:closed-form-sol-geodesic}
    Y(t) = Q \exp\!\left( \begin{bmatrix}
        \Omega   &   -K\tr \\
        K        &    O_{n-p}
    \end{bmatrix} t \right)
    \begin{bmatrix}
        I_{p} \\
        O_{(n-p)\times p}
    \end{bmatrix},
\end{equation}
with $ Q = \big[ Y_{0} \ Y_{0\perp} \big] $, $Y_{0\perp}$ being any matrix whose columns span $\mathcal{Y}_{0}^{\perp}=(\mathrm{span}(Y_{0}))^{\perp}$.
If $ t = 1 $, this is precisely the Riemannian exponential on the Stiefel manifold. In this paper, we denote by $ A $ the matrix in the argument of the matrix exponential $ \exp $.

\begin{remark}\label{rmk:Y_0perp}
The matrix $Y_{0\perp}$ does not need to be orthonormal. Indeed, its only requirement is that it has to span $\mathcal{Y}_{0}^{\perp}=(\mathrm{span}(Y_{0}))^{\perp}$, i.e., the orthogonal subspace to $\mathcal{Y}_{0}=\mathrm{span}(Y_{0})$. See appendix \ref{sec:freedom}. For the convenience of our analysis and implementation, we always assume that $Y_{0\perp}$ is orthonormal so that $ Q = \big[ Y_{0} \ Y_{0\perp} \big] $ is an orthogonal matrix.
\end{remark}

\section{Problem statement}\label{sec:problem_statement}

In this section, we state the problem more formally.
Given two points $Y_0$, $Y_1$ on $\Stnp$ that are sufficiently close to each other, finding the distance between them is equivalent to finding the tangent vector $\xi^{\ast} \in \mathrm{T}_{Y_0}\Stnp$ with the shortest possible length such that \protect{\cite{Lee:2018,boumal2023intromanifolds}}
\[
\Exp_{Y_0}(\xi^{\ast}) = Y_1,
\]
where $\Exp_{Y_0}$ denotes the Riemannian exponential mapping at $Y_0$. 
The solution to this problem is equivalent to the Riemannian logarithm of $Y_{1}$ with base point $Y_{0}$
\[
 \xi^{\ast} = \Log_{Y_0}(Y_1).
\]
We refer the reader to Figure~\ref{fig:single_shooting} for an illustration of the problem statement.

In terms of the differential equation \eqref{eq:ode_geodesics} governing the geodesic, the problem statement may be written as follows:

Find $ \xi^{\ast} \equiv \dt{Y}(0) \in \mathrm{T}_{Y_0}\Stnp$ such that the second-order ODE
\begin{equation}\label{eq:bvp}
\ddt{Y} = - \dt{Y}\dt{Y}\tr Y - Y \big( (Y\tr\dt{Y})^{2} + \dt{Y}\tr\dt{Y}\big), \quad
\text{with boundary conditions} \
\begin{cases}
Y(0)   = Y_{0}, \\
Y(1) = Y_{1},
\end{cases}
\end{equation}
is satisfied. This problem is known as a \emph{boundary value problem} (BVP).

\section{Single shooting method}\label{sec:single_shooting}

The single shooting is a classical numerical scheme for solving boundary value problems. The main idea is to reformulate the BVP as an initial value problem (IVP), guess the initial value of the acceleration, and then solve a nonlinear equation. It turns a BVP into a root-finding problem. The zeros of the nonlinear equation can be computed with any root-finding algorithm, but single shooting typically uses Newton's method, which enjoys quadratic convergence.

In this section, we give the details on how to apply the single shooting method to the endpoint geodesic problem on the Stiefel manifold.
We start by recasting the BVP \eqref{eq:bvp} into an IVP.
Let $Z_{1}(t) = Y(t) $, $Z_{2}(t) = \dt{Y}(t)$ denote the geodesic and its derivative, respectively, and let
\[
Z(t) = \begin{pmatrix}
Z_1(t) \\
Z_2(t)
\end{pmatrix}.
\]
We get the initial value problem (we omit the dependence on $t$)
\begin{equation}\label{eq:ivp}
\begin{split}
\dt{Z} = \begin{pmatrix}
\dt{Z}_1 \\
\dt{Z}_2
\end{pmatrix} = \begin{pmatrix}
Z_2 \\
- Z_2 Z_2\tr Z_1 - Z_1 \big( (Z_1\tr Z_2)^{2} + Z_2\tr Z_2 \big)
\end{pmatrix}, \\
\text{with initial conditions} \
Z(0) = \begin{pmatrix}
Z_1(0) \\
Z_2(0)
\end{pmatrix} =
\begin{pmatrix}
Y_0 \\
\xi
\end{pmatrix}.
\end{split}
\end{equation}
Here, $ \xi $ is the unknown such that $ Z_{1}(1) = Y_{1} $.
In practice, since we already have the explicit formula \eqref{eq:closed-form-sol-geodesic} for the geodesic $Z_1(t) $, we do not need to integrate the initial value problem \eqref{eq:ivp}.
The explicit formula for $Z_2$ is just the derivative of $Z_1$ with respect to $t$, namely,
\[
Z_2(t) = Q \, \exp\!\left(
\begin{bmatrix}
\Omega   &   -K\tr \\
K        &    O_{n-p}
\end{bmatrix} t \right)
\begin{bmatrix}
\Omega  \\
K     
\end{bmatrix}.
\]
Now let us define the function
\begin{equation}\label{eq:nonlinear-eq}
F(\xi) = \vecop\!\big( Z_1(1,\xi) - Y_1 \big),
\end{equation}
where we emphasize the dependence on $ \xi $.
Roughly speaking, this represents the mismatch between $ Z_1(1)$, i.e., the geodesic at $ t = 1$, and the boundary condition $Y_1$ we wish to enforce.
Our goal is to find $\xi^{\ast}$ such that
\[
   F(\xi^{\ast}) = 0.
\]
As mentioned above, this is a root-finding problem of a nonlinear equation, which can be solved by \emph{Newton's method}. To apply Newton's method, we need the Jacobian matrix of $F(\xi)$ with respect to $\xi$, denoted $J_{F}^{\xi}$.
This is actually $J_{Z_1}^{\xi}$, the Jacobian matrix of $Z_1$ with respect to $\xi$, since $Y_1$ appearing in $F(\xi)$ is not a function of $\xi$.

Here, we first give the algorithm, and then in the following sections, we will explain in more detail the derivation and the algorithmic components.
The pseudocode for the single shooting method on the Stiefel manifold is given in Algorithm~\ref{algo:simple-shooting}. As a stopping criterion, the norm of $ F $ is often used; in section~\ref{sec:SS_initial_guess}, we consider the norm of the residual $ \delta\xi^{(k)} $. 

\medskip

\begin{algorithm}
   \SetAlgoLined
   Given $Y_0$, $Y_1$\;
   \KwResult{$\xi^{\ast}$ such that $\Exp_{Y_0}(\xi^{\ast}) = Y_1$.}
   Compute the initial guess $ \xi^{(0)} $ (according to Algorithm~\ref{algo:initial-guess})\;
   Set $\xi^{(k)} = \xi^{(0)}$\;
   \While{a stopping criterion is met}{
      Compute Jacobian matrix $J_{Z_1}^{(k)}$ (see later~\eqref{eq:J_Z_1_x})\;
      Compute $F^{(k)} = \vecop\!\big( Z_1^{(k)} - Y_1 \big) $\;
      Solve $F^{(k)} + J_{Z_1}^{\xi^{(k)}}\, \delta\xi^{(k)} = 0$ for $\delta\xi^{(k)}$\;
      Update $\xi^{(k)} \leftarrow \xi^{(k)} + \delta\xi^{(k)}$\;
   }
   \caption{Single shooting on the Stiefel manifold.}\label{algo:simple-shooting}
\end{algorithm}

\subsection{Parametrization of the tangent space} \label{sec:ss_parametrization}

The tangent vector $\xi $ belongs to $ \R^{n \times p}$, but by inspecting its structure,
\[
   \xi = Y_{0} \Omega + Y_{0\perp} K,
\]
one can observe that it only depends on $ \bar{n} \coloneqq np - \tfrac{1}{2} (p+1)$ parameters (the dimension of the Stiefel manifold). Therefore, we can express $\xi$ as a function of these $ \bar{n} $ parameters. By standard linear algebra arguments, it is possible to find a matrix $ B \in \mathrm{R}^{p^2 \times \frac{1}{2} p (p-1)}$ whose columns form a basis of $\mathcal{S}_{\mathrm{skew}}$. This allows us to write the vectorization of $\Omega$ as
\[
\vecop(\Omega) = B s,
\]
for some $ s \in \R^{\frac{1}{2} p (p-1)}$ being a column vector representing $\Omega$ in the basis $B$ of $\mathcal{S}_{\mathrm{skew}}$.
The vectorization of the matrix $ K $ is simply $ k = \vecop(K) \in \R^{(n-p)p} $.
Hence, we can collect the coefficients of $\xi$ in a single vector
\[
x = \begin{pmatrix}
s \\
k
\end{pmatrix} \in \R^{\bar{n}}.
\]
Let us call
\begin{equation}\label{eq:matrix_A}
   A(x) = \begin{bmatrix}
             \Omega   &   -K\tr \\
             K        &   O_{n-p}
          \end{bmatrix}
\end{equation}
the matrix in the argument of the exponential appearing in the geodesic \eqref{eq:closed-form-sol-geodesic}. It is a function of $ x $ because the matrices $ \Omega $ and $ K $ are formed by the coefficients of the vector $ x $. Then~\eqref{eq:closed-form-sol-geodesic} can be rewritten as
\[
Z_1(1,x) = Q \exp ( A(x) )
\begin{bmatrix}
I_{p} \\
O_{(n-p)\times p}
\end{bmatrix}.
\]
Then~\eqref{eq:nonlinear-eq} becomes
\begin{equation}\label{eq:nonlinear-eq-x-explicit}
   F(x) = Z_1(1,x) - Y_1,
\end{equation}
where we highlight the dependence on $ x $ and omit the $ \vecop $ operator for readability.
Newton's method consists in solving successive linearizations of this equation, i.e.,
\begin{equation}\label{eq:pert-eq}
   F(x + \delta x) = Z_1(x + \delta x) - Y_1 = 0.
\end{equation}

Here, the term $Z_1(x + \delta x)$ is the expression for the geodesic when applying a small perturbation $ \delta x $ to the vector $ x $. From the expansion of $Z_1(x + \delta x)$ we will be able to read the Jacobian of $ Z_{1} $ with respect to $x$, denoted $J_{Z_1}^{x}$. Applying matrix perturbation theory, we obtain
\begin{equation}\label{eq:perturbation-geodesic}
Z_1(x+\delta x ) = Z_1(x) + 
Q \, \D \exp(A(x))\big[\D \! A(x)[\delta x]\big]
\begin{bmatrix}
I_{p} \\
O_{(n-p)\times p}
\end{bmatrix}
+ o(\Vert \delta x \Vert),
\end{equation}
where the notation $ \D \exp(A(x))\big[\D \! A(x)[\delta x]\big] $ denotes the Fr\'{e}chet derivative of the matrix exponential at $A(x)$ in the direction of $ \D \! A(x)[\delta x] $.
A chain rule is involved in this term, so we first need to find $ \D \! A(x)[\delta x] $. The perturbation of $A(x)$ yields
\[
   A(x+\delta x ) = A(x) +  \D \! A(x)[\delta x] + o(\Vert \delta x \Vert).
\]
Let $\blkvec$ be the operator that performs a block-wise vectorization of $A(x)$, namely,
\begin{equation*}
\blkvec( A(x) ) = \blkvec\!\left( \begin{bmatrix}
\Omega   &   -K\tr \\
K        &    O_{n-p}
\end{bmatrix} \right) = 
\begin{bmatrix}
\vecop(\Omega) \\
\vecop(K) \\
\vecop(-K\tr) \\
\vecop(O_{n-p})
\end{bmatrix}.
\end{equation*}
Using the vectorization of $ \Omega $ and $ K $ introduced above, we obtain
\begin{equation*}
   \blkvec( A(x) ) =
   \begin{bmatrix}
      B   &   O_{p^2\times p(n-p)} \\
      O_{p(n-p)\times \frac{1}{2}p(p-1)}   &   I_{p(n-p)} \\
      O_{p(n-p)\times \frac{1}{2}p(p-1)}   &  -\Pi_{n-p,p} \\
      O_{(n-p)^{2}\times \frac{1}{2}p(p-1)}   &   O_{(n-p)^{2}\times p(n-p)}
   \end{bmatrix}
   \begin{pmatrix}
      s \\
      k
   \end{pmatrix},
\end{equation*}
where $\Pi_{n-p,p}$ is the \emph{perfect shuffle matrix} defined by
\[
   \vecop(X\tr) = \Pi_{n-p,p} \, \vecop(X).
\]
From the last equation, we can identify the Jacobian matrix of $A(x)$ with respect to $x$ as
\begin{equation}\label{eq:J_A_x}
   J_{A(x)}^{x} = 
   \begin{bmatrix}
   B   &   O_{p^2\times p(n-p)} \\
   O_{p(n-p)\times \frac{1}{2}p(p-1)}   &   I_{p(n-p)} \\
   O_{p(n-p)\times \frac{1}{2}p(p-1)}   &  -\Pi_{n-p,p} \\
   O_{(n-p)^{2}\times \frac{1}{2}p(p-1)}   &   O_{(n-p)^{2}\times p(n-p)}
   \end{bmatrix}.
\end{equation}
Hence $ \vecop( \D \! A(x)[\delta x] ) = J_{A(x)}^{x} \, \delta x $. 
We still need a map that links the block-wise vectorization $ \blkvec $ to the ordinary column-stacking vectorization $ \vecop $. Since this mapping is linear, it can be represented by a matrix $ T \in \R^{n^2 \times n^2} $
\[
   \vecop( \D \! A(x)[\delta x] ) = T \cdot \blkvec(  \D \! A(x)[\delta x] ).
\]

Now we go back to the perturbation of the matrix exponential, whose expansion is
\[
   \exp( A + \delta A ) = \exp(A) + \D \exp(A) [\delta A] + o(\Vert \delta A \Vert),
\]
where $ \D \exp(A) [\delta A] $ is the Fr\'{e}chet derivative of the matrix exponential at $A$ in the direction of $ \delta A $. Vectorizing $ \D \exp(A) [\delta A] $ we get
\[
\vecop( \D \exp(A) [\delta A] ) = J_{\exp(A)}^{A} \vecop(\delta A),
\]
with $J_{\exp(A)}^{A}$ being the Jacobian of the matrix exponential. A closed-form expression for $J_{\exp(A)}^{A}$ is given in \protect{\cite{Higham:2008,Najfeld:1995}}, 
\begin{equation}\label{eq:J_expA_A}
   J_{\exp(A)}^{A} = \Big( \exp(A\tr/2) \otimes \exp(A/2) \Big) \, \sinch\!\left( \tfrac{1}{2} [ A\tr \oplus (-A)] \right),
\end{equation}
where $\oplus$ denotes the Kronecker sum: $A\tr \oplus (-A) = A\tr \otimes I_n - I_n \otimes A$, and $\sinch$ is the hyperbolic $\sinc$,
\[
   \sinch(y) = \sinh(y)/y.
\]
Vectorizing the second term on the right-hand side of \eqref{eq:perturbation-geodesic} and wrapping things up, we get
\begin{align*}
\vecop\!\left( Q \D \exp(A(x))\big[\D \! A(x)[\delta x]\big]
\begin{bmatrix}
I_{p} \\
O
\end{bmatrix} \right) & = \left( \big[ I_{p} \ \ O \big] \otimes Q \right)
\vecop\!\Big(\D \exp(A(x))\big[\D \! A(x)[\delta x]\big]\Big) \\
& = \left( \big[ I_{p} \ \ O \big] \otimes Q \right) J_{\exp(A)}^{A}\,\vecop(\D \! A(x)[\delta x])\\
& = \left( \big[ I_{p} \ \ O \big] \otimes Q \right) J_{\exp(A)}^{A}\,T\,\blkvec(\D \! A(x)[\delta x])\\
& = \left( \big[ I_{p} \ \ O \big] \otimes Q \right) J_{\exp(A)}^{A}\,T\, J_{A(x)}^{x} \, \delta x.
\end{align*}
From the last equation, we can identify the sought Jacobian matrix of $Z_1$ with respect to $x$, i.e.,
\begin{equation}\label{eq:J_Z_1_x}
   J_{Z_1}^{x} = \left( \big[ I_{p} \ \ O_{p\times (n-p)} \big] \otimes Q \right) J_{\exp(A)}^{A} \, T \, J_{A(x)}^{x}.
\end{equation}
Notice its dimension $ J_{Z_1}^{x} \in \R^{np \times \bar{n} } $.

Finally, the linearization of \eqref{eq:pert-eq} is
\[
Z_1(x) + J_{Z_1}^{x}\,\delta x - Y_1 = 0,
\]
i.e., the Newton update
\[
J_{Z_1}^{x}\,\delta x = - F(x).
\]
\begin{remark}
This is an \emph{overdetermined} system to be solved for $\delta x$.
Indeed, $ J_{Z_1}^{x} \colon \R^{\bar{n}} \to \R^{np} $, and since for all $ p \geqslant 1 $ one has $ np > \bar{n} $, there are always more equations than unknowns. The system is overdetermined, but Newton's equation has a solution since $ F(x) = 0 $ is assumed to have a solution, since we assume that there exists a geodesic connecting $ Y_{0} $ and $ Y_{1} $.
\end{remark}

\subsection{The initial guess}\label{sec:SS_initial_guess}

It is well known that Newton's method exhibits only local convergence properties, which means that the method requires a sufficiently good initial guess to converge.  Therefore, selecting a ``good enough'' initial guess is crucial. This section outlines our approach to initializing Newton's method, which involves choosing an initial guess $\xi^{(0)}$ that is close enough to $\xi^{\ast}$.
To this aim, we use a first-order approximation of the matrix exponential $ \exp(A) \approx I + A $ in \eqref{eq:nonlinear-eq-x-explicit} and solve for $\xi$. 
This yields the first-order approximation to the solution $\xi^{\ast}$ as
\[
   \bar{\xi} = Y_1 - Y_0.
\]
This is no longer an element of the tangent space, so we need to project it onto $ \mathrm{T}_{Y_0}\Stnp $. We expect it to be a satisfactory initial approximation to the sought tangent vector $\xi^{\ast}$.
We recall from section~\ref{sec:projectors} that the projection of a vector $\xi$ onto the tangent space to the Stiefel manifold at $Y$ is given by
\[
   \P_{Y} \! \xi = Y \mathrm{skew}(Y\tr\xi) + (I-YY\tr) \, \xi.
\]
The projection of $\bar{\xi}$ onto the tangent space at $Y_0$ is
\[
   \P_{Y_{0}} \! \bar{\xi} = Y_0 \,\mathrm{skew}\big(Y_0\tr(Y_1 - Y_0)\big) + (I_n - Y_0 Y_0\tr)(Y_1 - Y_0) = Y_1 - Y_0 \, \mathrm{sym}(Y_0\tr Y_1).
\]
To get $\xi^{(0)}$, we rescale this vector so that its norm is equal to the norm of $\bar{\xi}$, i.e.,
\[
   \xi^{(0)} = \frac{\left\Vert \bar{\xi} \right\Vert}{\left\Vert \P_{Y_{0}} \! \bar{\xi} \right\Vert} \, \P_{Y_{0}} \! \bar{\xi}.
\]
This procedure is summarized in Algorithm~\ref{algo:initial-guess} and illustrated in Figure~\ref{fig:SS_initial_guess}.

\begin{figure}[htbp]
   \centering
   \includegraphics[width=0.55\columnwidth]{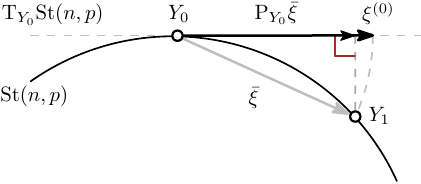}
   \caption{Initial guess for the single shooting method on the Stiefel manifold.}
   \label{fig:SS_initial_guess}
\end{figure}

\medskip

\begin{algorithm}
   \SetAlgoLined
   Given $Y_0$, $Y_1$\;
   Compute $ \bar{\xi} = Y_1 - Y_0 $\;
   Compute $ \P_{Y_{0}} \! \bar{\xi} = Y_1 - Y_0 \, \mathrm{sym}(Y_0\tr Y_1) $\;
   Compute $ \xi^{(0)} = \frac{\left\Vert \bar{\xi} \right\Vert}{\left\Vert \P_{Y_{0}} \! \bar{\xi} \right\Vert} \, \P_{Y_{0}} \! \bar{\xi} $\;
   Return  $ \xi^{(0)} $.
   \caption{Initial guess for the single shooting method on the Stiefel manifold.}\label{algo:initial-guess}
\end{algorithm}

\subsection{A smaller formulation}\label{sec:baby_problem}

It can be shown that the endpoint geodesic problem on $\Stnp$ is equivalent to an endpoint geodesic problem on $\mathrm{St}(2p,p)$ (see \protect{\cite{Edelman:1998,Rentmeesters:2013}}).
In the formulation above, the complexity of computing the matrix exponential is $O(n^3)$, but if $p \ll n $, then the smaller formulation can be used, and its computational cost is only $O(p^3)$. In practice, it makes sense to consider the formulation on $\mathrm{St}(2p,p)$ only if $ p < \tfrac{n}{2} $.
This section shows how this smaller formulation can be obtained.

Consider the same problem setting as in the previous sections, and let the QR factorization of $K$ be
\[
   K = \big[ Q \ \ Q_{\perp} \big]
      \begin{bmatrix}
         R \\ O_{(n-2p)\times p}
      \end{bmatrix} = Q R,
\]
where $\big[ Q \ \ Q_{\perp} \big] \in \R^{(n-p)\times (n-p)} $ is the orthogonal factor of $K$, with $Q \in \R^{(n-p)\times p}$ and $ Q_{\perp} \in \R^{(n-p)\times (n-2p)} $ orthonormal matrices, and $R\in\R^{p\times p}$ is upper triangular.

In appendix \ref{sec:baby} we show that
\begin{equation}\label{eq:baby}
Y_{1} = \big[ Y_{0} \ \ Y_{0\perp} Q \big] \, \exp\!\left( \begin{bmatrix} \Omega  &  -R\tr \\  R       &     O_p \end{bmatrix} \right)\begin{bmatrix}
I_{p} \\
O_{p}
\end{bmatrix}.
\end{equation}
Roughly speaking, the matrix $ Q $ can be pulled out from the matrix exponential, so that the matrix in $ \exp $ argument is only of size $2p$-by-$2p$.

Here, our aim is to find $ \Omega \in \R^{p\times p} $ and $ R \in \R^{p\times p} $ such that \eqref{eq:baby} holds true. Then, we can reconstruct vector as $\xi = Y_0 \Omega + Y_{0\perp} Q R $.

Let $Y_1$ be decomposed in the basis $\big[ Y_{0} \ \ Y_{0\perp} Q \big]$, and let $M$ and $N$ be the components of $Y_1$ in this basis
\begin{equation}\label{eq:Y1MN}
Y_1 = Y_0 M + Y_{0\perp} Q N.
\end{equation}
This implies that
\begin{equation}\label{eq:baby-problem}
\begin{bmatrix}
M\\
N
\end{bmatrix}=
\exp\!\left( \begin{bmatrix} \Omega  &  -R\tr \\  R       &     O_p \end{bmatrix} \right)\begin{bmatrix}
I_{p} \\
O_{p}
\end{bmatrix}.
\end{equation}
Left-multiplication of \eqref{eq:Y1MN} by $Y_0\tr $ and $Y_{0\perp}\tr $ yields, respectively $ Y_0\tr Y_1 = M$ and $ Y_{0\perp}\tr Y_1 = QN $. So one possible way to get $N$ out of $Y_{0\perp}\tr Y_1$ is to compute its QR factorization
\begin{equation}\label{eq:Y_0_perp_t_Y_1}
[ Q, N ] = \mathrm{qr}(Y_{0\perp}\tr Y_1).
\end{equation}

The remarkable observation is that \eqref{eq:baby-problem} describes an endpoint geodesic problem on $\mathrm{St}(2p,p)$ with base point
\[
\widehat{Y}_0 = \begin{bmatrix}
I_p \\
O_p
\end{bmatrix},
\]
with $\widehat{\xi} = \widehat{Y}_0 \Omega + \widehat{Y}_{0\perp} R$ the tangent vector to $\mathrm{St}(2p,p)$ at $\widehat{Y}_0$, and arrival point
\[
\widehat{Y}_1 = \begin{bmatrix}
M \\
N
\end{bmatrix}.
\]
Indeed, this problem setting yields the geodesic problem
\[
\widehat{Y}_1=
\underbrace{[ \widehat{Y}_0 \ \ \widehat{Y}_{0\perp}]}_{I_{2p}} \, \exp\!\left( \begin{bmatrix} \Omega  &  -R\tr \\  R       &     O_p \end{bmatrix} \right)\begin{bmatrix}
I_{p} \\
O_{p}
\end{bmatrix},
\]
which is exactly \eqref{eq:baby-problem}.

This problem can be solved via the single shooting method described above to find $\widehat{\xi}^{(k)} = \widehat{Y}_0 \Omega^{(k)} + \widehat{Y}_{0\perp} R^{(k)}$ at a given iteration $k$ (a stopping criterion is needed here). The components are given by $\Omega^{(k)} = \widehat{Y}_0\tr \widehat{\xi}^{(k)} $, $R^{(k)} = \widehat{Y}_{0\perp}\tr \widehat{\xi}^{(k)} $.
Finally, as promised, the tangent vector  $\xi^{(k)}$  of the original problem on $\Stnp$ can be recovered by
\[
\xi^{(k)} = Y_0 \Omega^{(k)} + Y_{0\perp} Q R^{(k)},
\]
where $Q\in\R^{(n-p)\times p}$ is the orthonormal factor of $Y_{0\perp}\tr Y_1$ as in \eqref{eq:Y_0_perp_t_Y_1}.

\subsection{Analysis of the Jacobian \texorpdfstring{$ J_{\exp(A)}^{A} $}{TEXT}}\label{sec:analysis_J_exp}

In this section, we state some more theoretical result about the Jacobian of the matrix exponential involved in the single shooting method.
As we did in section~\ref{sec:ss_parametrization}, let the Jacobian of $ \exp(A) $ with respect to $ A $ be as in~\eqref{eq:J_expA_A}, i.e.,
\[
   J_{\exp(A)}^{A} = \Big( \exp(A\tr/2) \otimes \exp(A/2) \Big) \, \sinch\!\left( \tfrac{1}{2} [ A\tr \oplus (-A)] \right).
\]
Since $A$ is normal, we can apply the theorems presented in appendix~\ref{app:sing_val_KX} to bound the singular values of $J_{\exp(A)}^{A}$.
We obtain the following lemma.

\begin{lemma}\label{lemma:sing_val_J_expA}
Let $A$ and $J_{\exp(A)}^{A}$ be as defined above, and let $ \alpha = \| A \|_{2} $. We have
\[
   \sigma_{\max}\big(J_{\exp(A)}^{A}\big) = 1  \qquad \text{and} \qquad  \sigma_{\min}\big(J_{\exp(A)}^{A}\big) = \left\vert \sinc \alpha \right\vert.
\]
\end{lemma}

\begin{proof}
   See appendix~\ref{app:proof_lemma_sing_val_J_expA}.
\end{proof}

Figure~\ref{fig:sinc_function} illustrates the function $ \left\vert \sinc \alpha \right\vert $ for $ \alpha $ in the interval $ [0,5] $. If $ \alpha = \pi $, then $ \sigma_{\min}\big(J_{\exp(A)}^{A}\big) = \left\vert \sinc \alpha \right\vert = 0 $, hence $ J_{\exp(A)}^{A} $ is singular. The figure also shows that in the interval $ [0, \pi ] $, $ \left\vert \sinc \alpha \right\vert $ is lower bounded by the straight line of equation $ 1 - \alpha/\pi $ (it is the tangent line to $ \left\vert \sinc \alpha \right\vert $ at $ \alpha = \pi $).

\begin{figure}[htbp]
  \centering
  \includegraphics[width=0.55\textwidth]{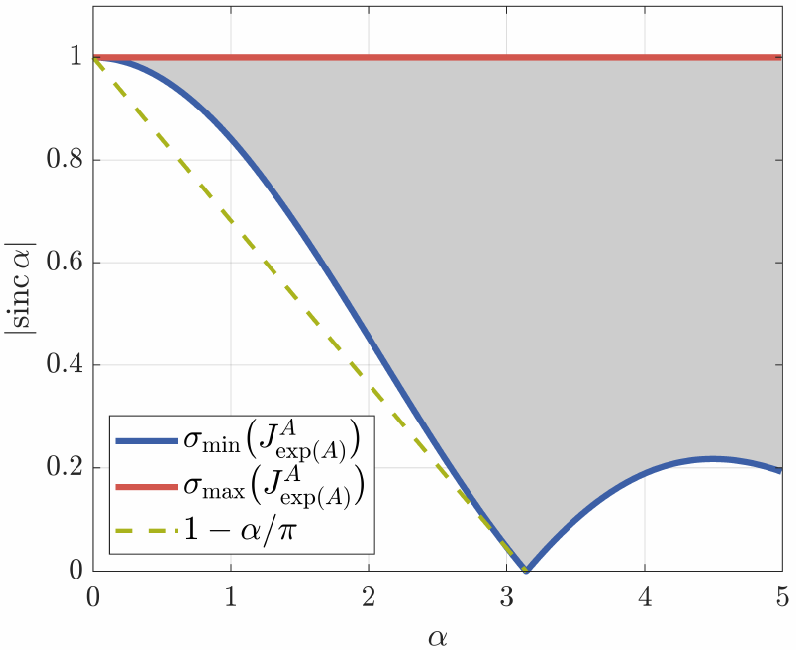}
  \caption{A plot of $ \left\vert \sinc \alpha \right\vert $ for $ \alpha $ in the interval $ [0,5] $.}\label{fig:sinc_function}
\end{figure}

\subsection{Connection between numerical experiments and existing Riemannian geometry results}\label{sec:connection}

In this section, we state some known results from Riemannian geometry and discuss how the numerical linear algebra from our analysis above reveals a nice connection with the cut locus, the injectivity radius, and the diffeomorphicity of the Riemannian exponential.

We first recall that in Euclidean geometry straight lines are also distance-minimizing curves. However, in Riemannian geometry, there exists a point $\gamma(t_{c})$, called \emph{cut point}, where the distance-minimizing property first breaks down \protect{\cite[Chapter~III]{Sakai:1996}}. The set $ C_{X} $ of these points along the geodesics emanating from $X$ is called the \emph{cut locus of} $X$.

The following standard result can be found in \protect{\cite[Lemma~5.7.9]{Petersen:2016}}, and is reported here for the reader's reference.

\begin{lemma}[\protect{\cite[Lemma~5.7.9]{Petersen:2016}}]\label{lemma:petersen}
   If $\xi_{1}$ is in the cut locus, then either
   \begin{enumerate}
      \item there exists another tangent vector $\xi_{2}$, different from $\xi_{1}$, such that
      \[
         \Exp_{X}(\xi_{1}) = \Exp_{X}(\xi_{2}),
      \]
      or
      \item $ \D \Exp_{X} $ is singular at $\xi_{1}$.
   \end{enumerate}
\end{lemma}

Figure~\ref{fig:cut_locus_unit_circle} illustrates the concept of cut locus for the unit circle $ \cS^{1} = \mathrm{St}(2,1) $. For $ \cS^{1} $, the cut locus of a point $X$ is the single point $Y$ opposite to it (the antipodal point), $ C_{X} = \lbrace Y \rbrace $. We say that $Y$ is the cut locus of $X$ in $ \cS^{1} $.
Clearly, the geodesic emanating from $X$ in the direction of $\xi_{1}$ (or $\xi_{2}$) stops being distance-minimizing at a distance of $\pi$ from $X$, i.e., at $Y$.
In other words, the antipodal point $Y$ to $ X $ is the point at which the Riemannian exponential $\Exp$ stops being a diffeomorphism. Indeed, $ Y $ can be reached by both $ \xi_{1} $ and $ \xi_{2} $, both having norm $ \pi $. This illustrates point 1. of Lemma~\ref{lemma:petersen}.

\begin{figure}[htbp]
  \centering
  \includegraphics[width=0.80\textwidth]{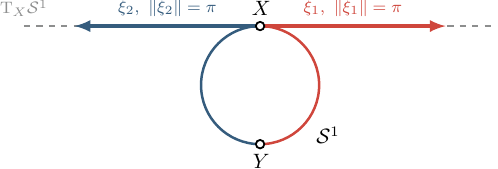}
  \caption{An illustration of the cut locus in the case of the unit circle $ \cS^{1} $.}\label{fig:cut_locus_unit_circle}
\end{figure}

More generally, from a numerical perspective, we recall that in step 7 of Algorithm~\ref{algo:simple-shooting} we need to solve the linear system
\[
   F_{x}^{(k)} = - J_{Z_1}^{x} \delta_{x}^{(k)}
\]
for $\delta_{x}^{(k)}$. In order to do this, the Jacobian $ J_{Z_1}^{x} $ needs to be a full-rank matrix. Here, we compute the (numerical) rank of this matrix, but we only focus on the special case of the unit sphere $ \cS^{n-1} $ in $ \R^{n} $, which is equivalent to $\mathrm{St}(n,1)$. We point out that, in this case, one has $ \| A \|_{2} = d(Y_{0},Y_{1}) $. Moreover, $ J_{Z_1}^{x} $ is of size $n$-by-$(n-1)$, so the Jacobian has full rank only if $ \rank \! \big( J_{Z_1}^{x} \big) = n-1 $, otherwise it is rank deficient.
Extensive numerical experiments for $p=1$ and varying $n$ suggest that there is a connection between the cut locus and $ \rank \! \big( J_{Z_1}^{x} \big) $. Figure~\ref{fig:rank_J_Z1_x_n20_p1_nrs_201} shows the values of $ \rank \! \big( J_{Z_1}^{x} \big) $ versus $d(Y_{0},Y_{1})$ for $p=1$ and $n=2,\dots,20$. Darker shades correspond to bigger values of $n$. We observe that for all the cases considered, the Jacobian is rank deficient only at $d(Y_{0},Y_{1}) = \pi$. This agrees with point 2. in Lemma~\ref{lemma:petersen}, and seems to be related to our result stated in Lemma~\ref{lemma:sing_val_J_expA} that $J_{\exp(A)}^{A}$ is singular when $ \| A \|_{2} $.

\begin{figure}[htbp]
  \centering
  \includegraphics[width=0.95\textwidth]{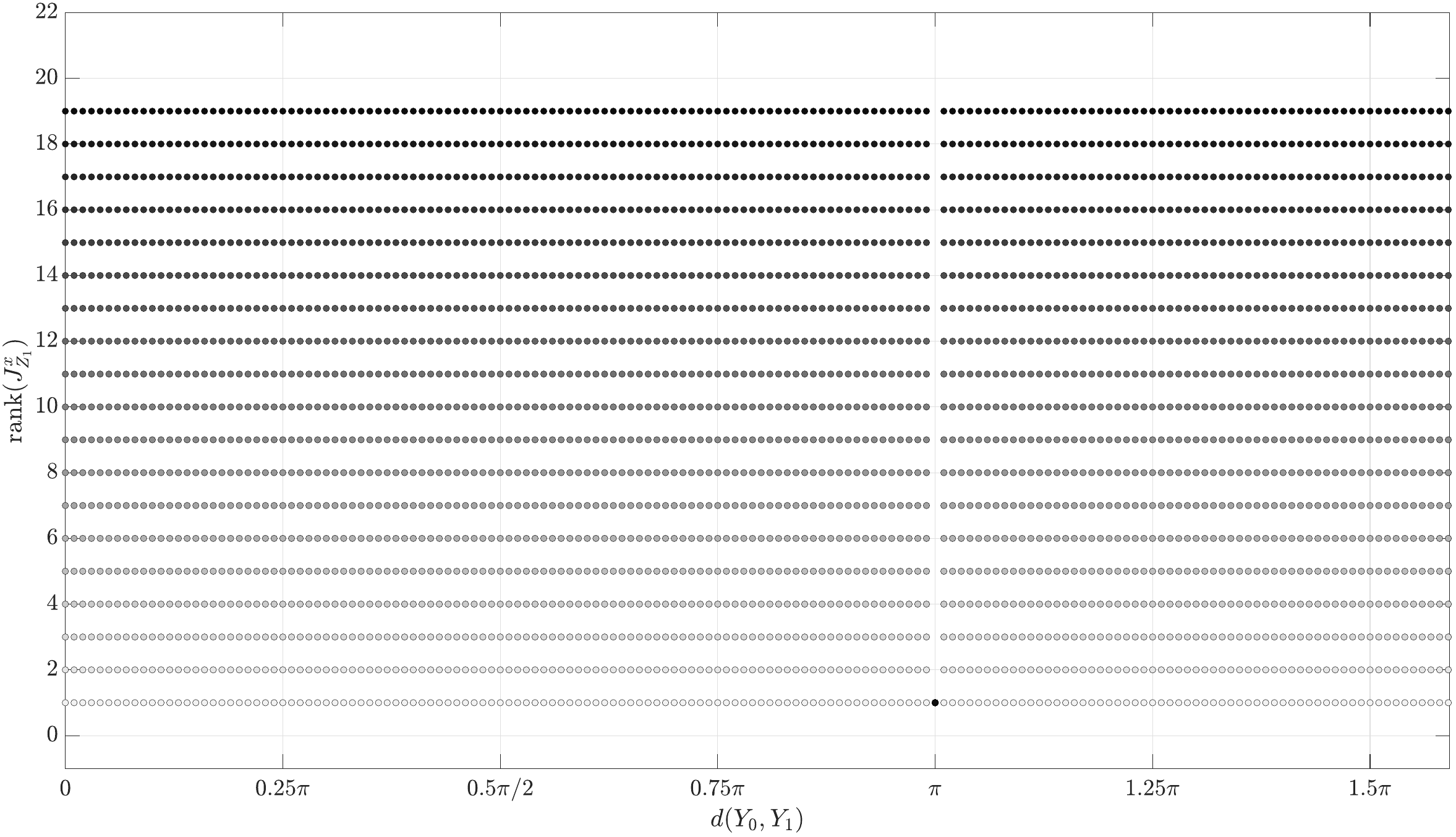}
  \caption{Values of $ \rank \! \big( J_{Z_1}^{x} \big) $ versus $ d(Y_{0},Y_{1}) $, for $ \cS^{n-1} $ with $ n = 2,\dots, 20 $.}\label{fig:rank_J_Z1_x_n20_p1_nrs_201}
\end{figure}

We also calculate the condition number of $ J_{Z_1}^{x} $ to gain more insight.
Figure~\ref{fig:cond_J_Z1_x_n20_p1_nrs_4001} shows the values of $ \cond \! \big( J_{Z_1}^{x} \big) $ versus $d(Y_{0},Y_{1})$ for $ \cS^{n-1} $ with $n=2,\dots,20$. Darker shades correspond to bigger values of $n$. The gray line is for $n=2$, while all the other lines for $n=3,\ldots,20$ overlap. In the case $n=2$, the Jacobian is always very well-conditioned. In all the other cases, we observe that $ J_{Z_1}^{x} $ becomes ill-conditioned as $d(Y_{0},Y_{1})$ approaches $\pi$, and the condition number eventually blows up at $d(Y_{0},Y_{1}) = \pi$.

\begin{figure}[htbp]
  \centering
  \includegraphics[width=0.95\textwidth]{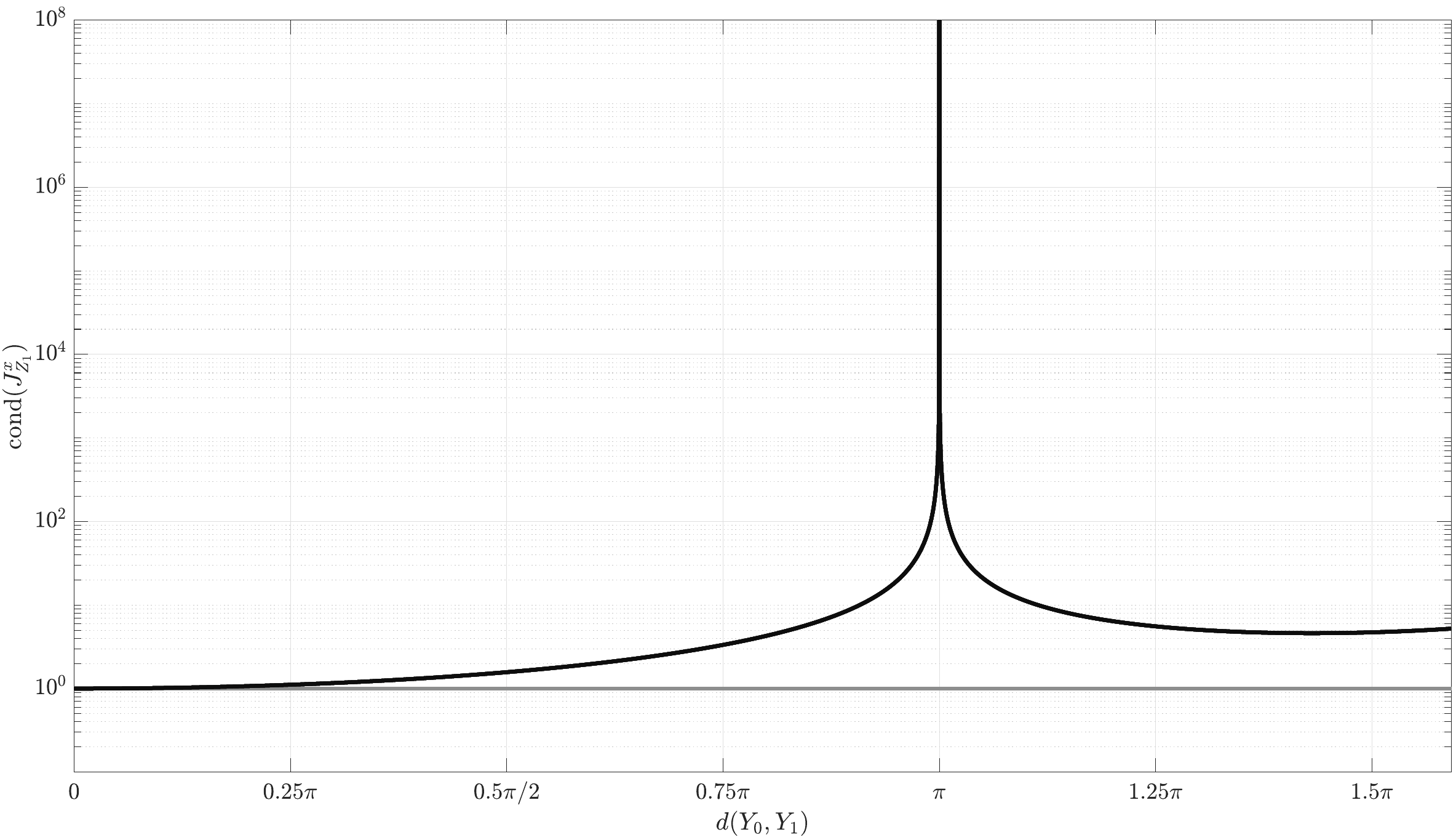}
  \caption{Values of $ \cond \! \big( J_{Z_1}^{x} \big) $ versus $ d(Y_{0},Y_{1}) $, for $ \cS^{n-1} $ with $ n = 2,\dots, 20 $. The gray line corresponds to $n=2$, while all the other lines overlap.}\label{fig:cond_J_Z1_x_n20_p1_nrs_4001}
\end{figure}

The cut locus is also closely related to the differentiability of the distance function. A standard result says that if $Y \in C_{X}$, then there exist at least two minimizing geodesics joining $X$ to $Y$ (see again Figure \ref{fig:cut_locus_unit_circle} for $\S^{1}$), and the distance function $d_{X} \colon \cM \to \R $ defined by $d_{X}(Y) \coloneqq d(X,Y)$ is not differentiable at $Y$ \protect{\cite[Chapter~III, Proposition~4.8(2)]{Sakai:1996}}.

Another important concept from Riemannian geometry, which relates to the cut locus and our numerical experiments, is the concept of injectivity radius.
The \emph{injectivity radius} $\inj_{X}(\cM)$ at $X$ is defined as the radius of the open ball in which the Riemannian exponential map $\Exp_{X}$ is a diffeomorphism from the tangent space $\mathrm{T}_{X}\cM$ to the manifold $\cM$ \protect{\cite[Chapter~III, Definition~4.12]{Sakai:1996}}. The \emph{global injectivity radius} is the infimum of the injectivity radius at $X$ over all points of the manifolds, i.e.,
\[
   \inj(\cM) \coloneqq \inf_{X \in \cM} \inj_{X}(\cM).
\]
The important property \protect{\cite[Chapter~III, Proposition~4.13(1)]{Sakai:1996}} tells us that the injectivity radius at $X$ is the distance from $X$ to the cut locus, i.e.,
\[
   \inj_{X}(\cM) = d(X,C_{X}).
\]
Hence for $\cS^{n-1}$, we have $ \inj_{X}(\cS^{n-1}) = \pi $ and $ \inj(\cS^{n-1}) = \pi $.
More generally, we have the following estimates on the injectivity radius. Let $\cM$ be a compact Riemannian manifold with everywhere strictly positive sectional curvature $K$, and let $C$ be an upper bound on $K$, i.e., $0<K \leq C$. The distance between a point and its cut locus is at least $\pi/\sqrt{C}$~\protect{\cite[Theorem~1(b)]{Klingenberg:1959}}, i.e.,
\[
   \inj_{X}(\cM) \geq \frac{\pi}{\sqrt{C}}.
\]
For the Stiefel manifold, an upper bound on its sectional curvature is given by $5/4$ \protect{\cite[p.~95]{Rentmeesters:2013}}, which implies the estimate
\[
   \inj(\Stnp) \geq \frac{2\sqrt{5}}{5} \, \pi \approx 0.8944 \, \pi.
\]
In the numerical experiments of section~\ref{sec:numerical_experiments}, we often use this value as a reference to (approximately) know whether a point is located outside or inside the injectivity radius of the Stiefel manifold $\Stnp$.

\section{Multiple shooting method}\label{sec:multiple_shooting}

A way to improve over single shooting is to consider a partition of the original interval into many smaller subintervals, which leads us to the \emph{multiple shooting method}.
This slicing permits to reduce the nonlinearity of the problem and improves numerical stability.
As in single shooting, Newton's method is also behind multiple shooting. The difference is that many initial value problems are solved separately on all multiple shooting intervals. The resulting system to be solved is larger, but the banded structure of the Jacobian matrix can be exploited to improve efficiency. A thorough description of the multiple shooting method can be found in \protect{\cite[\S 7.3.5]{Stoer:1991}}. Here, we will specialize the method in the context of the geodesic problem on the Stiefel manifold $ \Stnp $.

Let $ X $, $ Y $ be two points on a Stiefel manifold $ \Stnp $.
Consider a \emph{piecewise} (or \emph{broken}) \emph{geodesic} joining $ X $ to $ Y $, having $ m-1 $ geodesic segments. By this we mean that each curve segment is a geodesic on each subinterval.

Figure~\ref{fig:Broken_geodesic} provides an illustration of a broken geodesic.

\begin{figure}
   \centering
   \includegraphics[width=0.70\columnwidth]{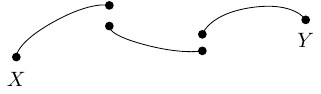}
   \caption{A schematic depiction of a broken geodesic.}
   \label{fig:Broken_geodesic}
\end{figure}

Let $\Sigma^{(k)}_{1}$ denote the point on the Stiefel manifold on the $k$th subinterval, and $\Sigma^{(k)}_{2}$ the tangent vector to $ \Stnp $ at $\Sigma^{(k)}_{1}$. Let $ \Sigma $ be the variable that collects the points and the tangent vectors for all the subintervals, namely,
\begin{equation*}
    \Sigma = \Big[\,
        \Sigma^{(1)}_{1} \quad
        \Sigma^{(1)}_{2} \quad
        \Sigma^{(2)}_{1} \quad
        \Sigma^{(2)}_{2} \quad
        \cdots \quad
        \Sigma^{(m-1)}_{1} \quad
        \Sigma^{(m-1)}_{2} \quad
        \Sigma^{(m)}_{1} \quad
        \Sigma^{(m)}_{2}\,
    \Big]\tr.
\end{equation*}
The compatibility conditions of the geodesic and its first derivative, plus the two boundary conditions denoted by $ r_{1} $ and $ r_{2} $, can be encoded into a system of nonlinear equations to be solved for $ \Sigma $
\begin{equation}\label{eq:MS_nonlinear_system}
    F(\Sigma ) = \begin{bmatrix}
        Z^{(1)}_{1} - \Sigma^{(2)}_{1} \\[4pt]
        Z^{(1)}_{2} - \Sigma^{(2)}_{2} \\[4pt]
        Z^{(2)}_{1} - \Sigma^{(3)}_{1} \\[4pt]
        Z^{(2)}_{2} - \Sigma^{(3)}_{2} \\[4pt]
        \vdots \\[4pt]
        Z^{(m-1)}_{1} - \Sigma^{(m)}_{1} \\[4pt]
        Z^{(m-1)}_{2} - \Sigma^{(m)}_{2} \\[4pt]
        r_1 = \Sigma^{(1)}_{1} - X \\[4pt]
        r_2 = \Sigma^{(m)}_{1} - Y
    \end{bmatrix}  = 0.
\end{equation}
Here, as in~\eqref{eq:ivp}, $ Z^{(k)}_{1} $ denotes the geodesic, whereas $ Z^{(k)}_{2} $ is the derivative of the geodesic with respect to $t$.
All the quantities $Z^{(k)}_{i}$ and $\Sigma^{(k)}_{i}$, $ k = 1, \ldots, m-1 $ and $ i = 1, 2 $, are to be understood as \emph{vectorized} quantities. 

Figure~\ref{fig:multiple_shooting} illustrates the variables (points and tangent vectors) involved in the multiple shooting on the Stiefel manifold.

\begin{figure}[htbp]
\centering
\includegraphics[width=0.65\columnwidth]{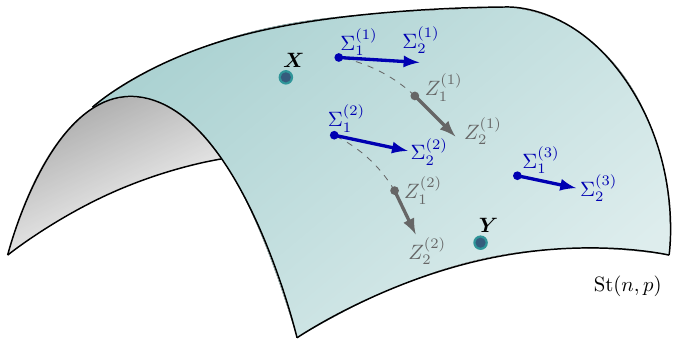}
\caption{Multiple shooting on the Stiefel manifold.}
\label{fig:multiple_shooting}
\end{figure}

Now consider the perturbed system
\begin{equation*}
    F( \Sigma + \delta \Sigma ) = 0, \qquad \text{with} \qquad
    \delta\Sigma = \Big[\,
        \delta\Sigma^{(1)} \quad
        \delta\Sigma^{(2)} \quad
        \cdots \quad
        \delta\Sigma^{(m)}\,
    \Big]\tr.
\end{equation*}
A linearization of the previous equation gives
\begin{equation}\label{eq:MS_linear_system}
    F(\Sigma ) + J_{F}^{\Sigma} \cdot \delta \Sigma = 0,
\end{equation}
where $J_{F}^{\Sigma} \in \R^{2mnp \times 2mnp}$ is a block Jacobian matrix. Each block $J_{F k\ell}^{\Sigma}\in \mathbb{R}^{np \times np}$ is given by
\begin{equation*}
    J_{F kk}^{\Sigma} = G^{(k)}, \qquad J_{F k,k+1}^{\Sigma} = -I_{2np}, \qquad k = 1, \ldots, m-1,
\end{equation*}
\begin{equation*}
    J_{F m,1}^{\Sigma} = C, \qquad J_{F m,m}^{\Sigma} = D, \qquad J_{F k\ell}^{\Sigma} = O_{2np} \quad \text{otherwise}.
\end{equation*}
Every $G^{(k)}$ is itself a Jacobian matrix for each subinterval defined as
\begin{equation}\label{eq:MS_jacobian_G_k}
    G^{(k)} = \begin{bmatrix}
        \frac{\partial Z^{(k)}_1}{\partial \Sigma^{(k)}_{1}}  & &  \frac{\partial Z^{(k)}_1}{\partial \Sigma^{(k)}_{2}} \\[12pt]
        \frac{\partial Z^{(k)}_2}{\partial \Sigma^{(k)}_{1}}  & &  \frac{\partial Z^{(k)}_2}{\partial \Sigma^{(k)}_{2}}
    \end{bmatrix} = \begin{bmatrix}
        J_{Z_1}^{\Sigma_{1}}  &  J_{Z_1}^{\Sigma_{2}} \\[10pt]
        J_{Z_2}^{\Sigma_{1}}  &  J_{Z_2}^{\Sigma_{2}}
    \end{bmatrix},
\end{equation}
where we omitted the superscript $^{(k)}$ in the last matrix for ease of notation. We refer the reader to appendix~\ref{app:multiple-shooting} for the explicit expressions of the Jacobian matrices appearing in \eqref{eq:MS_jacobian_G_k}.
The Jacobian matrices associated with the boundary conditions are given by
\begin{equation*}
    C = \begin{bmatrix}
        \frac{\partial r_1}{\partial \Sigma^{(1)}_1}  &   \frac{\partial r_1}{\partial \Sigma^{(1)}_2} \\[12pt]
        \frac{\partial r_2}{\partial \Sigma^{(1)}_1}  &   \frac{\partial r_2}{\partial \Sigma^{(1)}_2}
    \end{bmatrix} = \begin{bmatrix}
        I_{np}  &   O_{np} \\[4pt]
        O_{np}  &   O_{np}
    \end{bmatrix}, 
    \qquad
    D = \begin{bmatrix}
        \frac{\partial r_1}{\partial \Sigma^{(m)}_1}  &   \frac{\partial r_1}{\partial \Sigma^{(m)}_2} \\[12pt]
        \frac{\partial r_2}{\partial \Sigma^{(m)}_1}  &   \frac{\partial r_2}{\partial \Sigma^{(m)}_2}
    \end{bmatrix} = \begin{bmatrix}
        O_{np}  &   O_{np} \\[4pt]
        I_{np}  &   O_{np} 
    \end{bmatrix}.
\end{equation*}

\subsection{The initial guess}

To initialize the multiple shooting algorithm, we use the leapfrog method of Noakes~\cite{Noakes:1998}. The main idea behind this algorithm is to exploit the success of single shooting by subdividing the original problem into several subproblems, introducing intermediate points between $X$ and $Y$, for which the endpoint geodesic problem can be solved by the single shooting method.
The algorithm then iteratively updates a piecewise geodesic to obtain a globally smooth geodesic between $X$ and $Y$. The leapfrog algorithm resembles the multiple shooting method because they both partition the original interval into smaller subintervals.

It is challenging to say something about the global convergence of Newton's method. For local convergence, we have the result of the \emph{Newton--Kantorovich theorem}. In practical applications, a sufficient number of iterations in the leapfrog algorithm produces an iterate $ \Sigma^{(k)} $, which satisfies the conditions of the Newton--Kantorovich theorem. For this reason, we use leapfrog to initialize the multiple shooting method. We name the resulting algorithm LFMS and will illustrate it in more detail through the numerical experiments of section~\ref{sec:LFMS_num_experiments}.

\section{Numerical experiments}\label{sec:numerical_experiments}

In this section, we present some simple numerical experiments about the single shooting and the multiple shooting algorithms, and we report on their convergence behavior.
The algorithms were implemented in MATLAB and are publicly available at~\url{https://github.com/MarcoSutti/LFMS_Stiefel}. We conducted our experiments on a laptop Lenovo ThinkPad T460s with Ubuntu 23.04 LTS and MATLAB R2022a installed, with Intel Core i7-6600 CPU, 20GB RAM, and Mesa Intel HD Graphics 520.

\subsection{Numerical experiments for single shooting}\label{sec:SS_num_example}

We use the numerical experiments to test the convergence of the single shooting method and relate it to the analysis above. We consider the Stiefel manifold $ \mathrm{St}(15,p) $ with $p$, ranging from 1 to 15. As for the endpoints, we fix one point $ X = [ I_{p} \ \ O_{p \times (15-p)} ]\tr $, while the other point $ Y $ is placed at a prescribed distance $d(X,Y)$ from $ X $. We vary the distance $d(X,Y)$ from $ 0.85\,\pi $ to $ \pi $ (see Table~\ref{tab:SS_convergence_n15}). By using single shooting, we want to recover this distance. We expect that if $ d(X,Y) $ is less than the lower bound on the injectivity radius, which for the Stiefel manifold is lower bounded by $ 0.8944 \, \pi $ (as explained in section~\ref{sec:connection}), single shooting will converge. However, as $d(X,Y)$ increases, the algorithm might diverge, especially when $d(X,Y)$ is greater than the value $ 0.8944 \, \pi $.

Table~\ref{tab:SS_convergence_n15} reports on the convergence behavior of the single shooting method for the endpoint geodesic problem on  $\mathrm{St}(15,p)$ as a function of the distance $d(X,Y)$. The checkmark signifies that the algorithm converges, while the cross indicates that the algorithm diverges. We consider that single shooting converged when the norm of the residual $ \| \delta\xi^{(k)} \|_{2} $, as it appears in Algorithm~\ref{algo:simple-shooting}, is smaller than $ 10^{-10} $. We observe that in some cases, for $p=2$, 3, 4, and 5, for sufficiently large $ d(X,Y) $, the single shooting method does not converge.

\begin{table}[htbp]
   \caption{Single shooting on the Stiefel manifold $\mathrm{St}(15,p)$, for $p=1,\dots,15$.}
   \label{tab:SS_convergence_n15}
   \begin{center}
      \begin{tabular}{c|c|c|c|c|c|c|c}  
         \toprule
           $ n = 15 $ &  \multicolumn{7}{c}{$ d(X,Y) $}  \\
         \midrule
          $ p $   &  $ 0.85\,\pi $ &  $ 0.875\,\pi $  & $ 0.90\,\pi $  &  $ 0.925\,\pi $  & $ 0.95\,\pi $  &  $ 0.975\,\pi $ &    $ \pi $  \\
         \midrule
            1     &   \checkmark   &  \checkmark      &  \checkmark    &   \checkmark     &   \checkmark   &   \checkmark    &   \checkmark \\
            2     &   \checkmark   &  \checkmark      &  \xmark        &   \xmark         &   \xmark       &   \xmark        &   \xmark     \\
            3     &   \checkmark   &  \checkmark      &  \checkmark    &   \checkmark     &   \xmark       &   \xmark        &   \xmark     \\
            4     &   \checkmark   &  \checkmark      &  \checkmark    &   \checkmark     &   \xmark       &   \xmark        &   \xmark     \\
            5     &   \checkmark   &  \checkmark      &  \checkmark    &   \checkmark     &   \checkmark   &   \xmark        &   \xmark     \\
            6     &   \checkmark   &  \checkmark      &  \checkmark    &   \checkmark     &   \checkmark   &   \checkmark    &   \checkmark \\
            7     &   \checkmark   &  \checkmark      &  \checkmark    &   \checkmark     &   \checkmark   &   \checkmark    &   \checkmark \\
            8     &   \checkmark   &  \checkmark      &  \checkmark    &   \checkmark     &   \checkmark   &   \checkmark    &   \checkmark \\
            9     &   \checkmark   &  \checkmark      &  \checkmark    &   \checkmark     &   \checkmark   &   \checkmark    &   \checkmark \\
           10     &   \checkmark   &  \checkmark      &  \checkmark    &   \checkmark     &   \checkmark   &   \checkmark    &   \checkmark \\
           11     &   \checkmark   &  \checkmark      &  \checkmark    &   \checkmark     &   \checkmark   &   \checkmark    &   \checkmark \\
           12     &   \checkmark   &  \checkmark      &  \checkmark    &   \checkmark     &   \checkmark   &   \checkmark    &   \checkmark \\
           13     &   \checkmark   &  \checkmark      &  \checkmark    &   \checkmark     &   \checkmark   &   \checkmark    &   \checkmark \\
           14     &   \checkmark   &  \checkmark      &  \checkmark    &   \checkmark     &   \checkmark   &   \checkmark    &   \checkmark \\
           15     &   \checkmark   &  \checkmark      &  \checkmark    &   \checkmark     &   \checkmark   &   \checkmark    &   \checkmark \\
         \bottomrule
      \end{tabular} 
   \end{center}
\end{table}

For the particular case of the unit sphere, we observe from our numerical experiments that when we want to recover a distance that is bigger than $\pi$, the single shooting method will not recover the original tangent vector, but it will compute the tangent vector corresponding to the distance-minimizing geodesic. Indeed, the uniqueness of the geodesic connecting two points is guaranteed only inside the injectivity radius. For instance, for $\mathrm{St}(15,1)$, with $d(X,Y)= 1.10 \, \pi \approx 3.46 $, the single shooting method computes a tangent vector whose length (canonical norm) is $ \| \xi \|_{\mathrm{c}} \approx 2.83 $.

This happens also in other instances of the Stiefel manifold, not only for the unit sphere. For example, for $\mathrm{St}(15,2)$, with $d(X,Y)= 1.05 \, \pi \approx 3.30 $, the single shooting method computes a tangent vector whose length (canonical norm) is $ \| \xi \|_{\mathrm{c}} \approx 2.98 $.

To plot the convergence behavior of the single shooting algorithm, let us consider the Stiefel manifold $ \mathrm{St}(15,4) $. We fix one point $ X = [ I_{4} \ \ O_{4 \times 11} ]\tr $, while the other point $ Y $ is placed at a distance $ d(X,Y) = 0.75\,\pi $ from $ X $. In this way, the points $X$ and $Y$ lie at a distance that is for sure than the injectivity radius of $ \Stnp $, which is lower bounded by $ 0.8944 \, \pi $. This guarantees the existence and uniqueness of the minimizing geodesic between $X$ and $Y$. As marked in Table~\ref{tab:SS_convergence_n15}, we already know that single shooting converges in this case.

Figure~\ref{fig:Convg_ss_15_4} reports on the convergence behavior of the norm of the residual $ \| \delta\xi^{(k)} \|_{2} $. In the previous experiments, we stopped the algorithm when a tolerance $ 10^{-10} $ is reached, but here, to show the plateau at the double precision machine epsilon $ \epsmach \approx 2.22 \times 10^{-16} $, we run it for a few more iterations.
We observe that single shooting converges in five iterations with a quadratic convergence behavior.

\begin{figure}[htbp]
  \centering
  \includegraphics[width=0.55\textwidth]{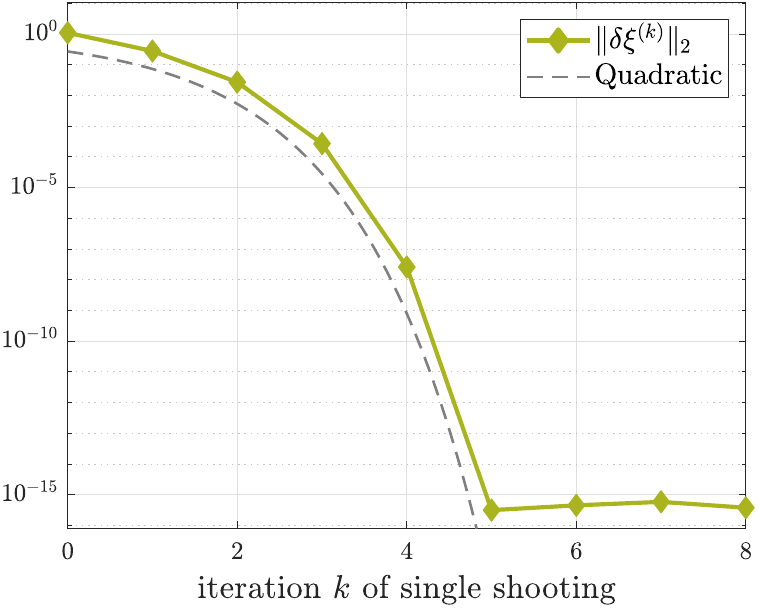}
  \caption{Convergence of the norm of the residual $ \| \delta\xi^{(k)} \|_{2} $ for single shooting on $ \mathrm{St}(15,4) $.}\label{fig:Convg_ss_15_4}
  \end{figure}

Table~\ref{tab:time_cost_SS} reports on the number of iterations, computational time, and monitored quantities for $\mathrm{St}(15,p)$, for $p=1,\ldots, 15 $ for $ d(X,Y) = 0.75\, \pi $. The stopping tolerance was set to $ 10^{-13} $. For each row, the same experiment is repeated and the computational time is averaged over 20 runs.

\begin{table}[htbp]
   \caption{Numerical results for the single shooting method on $\mathrm{St}(15,p)$, for $p=1,\ldots, 15 $ for a prescribed $ d(X,Y) = 0.75\,\pi $. The time is averaged over 20 runs.}
   \label{tab:time_cost_SS}
   \begin{center}
      \begin{tabular}{c|c|c|c|c}
         \toprule
                        $p$    & Iterations &    Time    & $ \| \delta\xi^{(\text{end})} \|_{2} $ & $ \| F^{(\text{end})} \|_{2} $  \\
         \midrule
                         1     &         4  &   0.00024  &                        0      &         0                      \\
                         2     &         6  &   0.00236  &   $ 1.5344 \times 10^{-15} $  &  $ 8.9422 \times 10^{-16} $    \\
                         3     &         6  &   0.00401  &   $ 8.0197 \times 10^{-16} $  &  $ 5.3533 \times 10^{-16} $    \\
                         4     &         6  &   0.00489  &   $ 7.1056 \times 10^{-16} $  &  $ 7.6902 \times 10^{-16} $    \\
                         5     &         6  &   0.00817  &   $ 5.6917 \times 10^{-16} $  &  $ 6.8165 \times 10^{-16} $    \\
                         6     &         6  &   0.01354  &   $ 1.2256 \times 10^{-15} $  &  $ 9.0419 \times 10^{-16} $    \\
                         7     &         6  &   0.02143  &   $ 1.1981 \times 10^{-15} $  &  $ 1.1402 \times 10^{-15} $    \\
                         8     &         6  &   0.02790  &   $ 7.6053 \times 10^{-16} $  &  $ 8.9126 \times 10^{-16} $    \\
                         9     &         6  &   0.03054  &   $ 5.8892 \times 10^{-16} $  &  $ 7.4760 \times 10^{-16} $    \\
                        10     &         6  &   0.03128  &   $ 6.9054 \times 10^{-16} $  &  $ 8.8569 \times 10^{-16} $    \\
                        11     &         6  &   0.03344  &   $ 6.8548 \times 10^{-16} $  &  $ 8.5836 \times 10^{-16} $    \\
                        12     &         5  &   0.02837  &   $ 5.0750 \times 10^{-14} $  &  $ 5.8053 \times 10^{-14} $    \\
                        13     &         5  &   0.03090  &   $ 7.7702 \times 10^{-16} $  &  $ 1.0236 \times 10^{-15} $    \\
                        14     &         5  &   0.03148  &   $ 8.2079 \times 10^{-16} $  &  $ 1.1232 \times 10^{-15} $    \\
                        15     &         4  &   0.02544  &   $ 4.9247 \times 10^{-16} $  &  $ 9.2350 \times 10^{-15} $    \\
         \bottomrule
      \end{tabular}
   \end{center}
\end{table}

From Table~\ref{tab:time_cost_SS} it appears that the single shooting method is very efficient for small values of $p$. The computational time is always smaller than 0.04 seconds, and it converges in at most 6 iterations. In the special cases of the sphere and the orthogonal group, it converges in 4 iterations. Nonetheless, the single shooting method scales very badly with respect to $p$, as it becomes very expensive as $p$ grows, as showed in the following Table~\ref{tab:time_cost_SS_n1000}.

\begin{table}[htbp]
   \caption{Numerical results for the single shooting method on $\mathrm{St}(1000,p)$, for several values of $p$, for a prescribed $ d(X,Y) = 0.75\,\pi $.}
   \label{tab:time_cost_SS_n1000}
   \begin{center}
      \begin{tabular}{c|c|c|c|c}
         \toprule
                        $p$      & Iterations &    Time    & $ \| \delta\xi^{(\text{end})} \|_{2} $ & $ \| F^{(\text{end})} \|_{2} $  \\
         \midrule
                         10      &         5  &   0.07859  &   $ 8.8690 \times 10^{-16} $  &  $ 7.9265 \times 10^{-16} $    \\
                         20      &         5  &   3.06617  &   $ 7.5946 \times 10^{-16} $  &  $ 8.5160 \times 10^{-16} $    \\
                         30      &         5  &   33.6876  &   $ 7.8204 \times 10^{-16} $  &  $ 1.0280 \times 10^{-15} $    \\
                         40      &         5  &   182.690  &   $ 8.9332 \times 10^{-16} $  &  $ 1.0394 \times 10^{-15} $    \\
                         50      &         5  &   713.083  &   $ 1.0277 \times 10^{-15} $  &  $ 1.1707 \times 10^{-15} $    \\
                         60      &         5  &   2145.06  &   $ 1.2720 \times 10^{-15} $  &  $ 1.3564 \times 10^{-15} $    \\
         \bottomrule
      \end{tabular}
   \end{center}
\end{table}

Indeed, the bottleneck of this algorithm is the calculation of the Jacobian $ J_{\exp(A)}^{A} $, which is a dense matrix of size $n^{2} \times n^{2}$ (or $ 4p^{2} \times 4p^{2} $ if $p<n/2$, since in that case we use the smaller formulation described in section~\ref{sec:baby_problem}). 
This is not a big issue in the practical applications of section~\ref{sec:applications} because $p$ stays moderately small (the biggest value is $p=20$ in section~\ref{sec:interp_on_stiefel}).

However, this can be seen as the price to pay for computing an explicit expression of the Jacobian matrix $ J_{\exp(A)}^{A} $ and wanting a quadratic convergence. Yet it might be possible in some cases to approximate this Jacobian $ J_{\exp(A)}^{A} $ and obtain a computationally cheaper algorithm. For example, if $ \| A \|_{2} $ is small enough, we may approximate $ J_{\exp(A)}^{A} $ it by an identity matrix. Of course, we will lose in terms of order of convergence since it will not be quadratic anymore, but we will save time and memory storage. This option will be explored in more detail in a future work.

\subsection{Leapfrog and multiple shooting (LFMS)}\label{sec:LFMS_num_experiments}

In this section, we provide an example of using the leapfrog algorithm of Noakes~\protect{\cite{Noakes:1998}} in combination with our multiple shooting method. The resulting algorithm is called LFMS. The combination of these two algorithms is also a novelty.

From an algorithmic point of view, we propose the following scheme, summarized in the flowchart of Figure~\ref{fig:Flowchart_Stiefel_Log}:
\begin{itemize}
\item Given two points $X$ and $Y$, for which we want to compute $d(X,Y)$, the first attempt to solve the endpoint geodesic problem is always done with single shooting.
\item If single shooting converges\footnote{In the experiments, we choose 10 as the maximum number of single shooting iterations. We consider that the single shooting fails when this number is exceeded.}, then the problem is solved, and we are done. If single shooting does not converge, we start with leapfrog with two subintervals, i.e., with $ m=3 $ points, i.e., the smallest partition possible.
\item If leapfrog with two subintervals does not work, i.e., if the single shooting behind leapfrog does not work, we keep increasing the number of subintervals until it converges. The single shooting behind leapfrog has to converge on each subinterval.
\item When leapfrog works, we perform a few iterations and then use the iterate found by leapfrog as an initial guess for multiple shooting. The problem is then solved with multiple shooting, which converges quadratically to the solution.
\end{itemize}

\begin{figure}[htbp]
\centering
\includegraphics[width=0.5\textwidth]{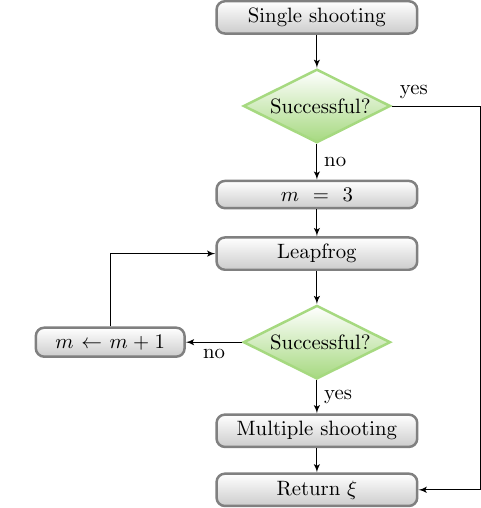}
\caption{Flowchart of the Stiefel Log algorithm.}\label{fig:Flowchart_Stiefel_Log}
\end{figure}

Consider the Stiefel manifold $ \mathrm{St}(12,3) $ as a concrete example. We fix one point $ X = [ I_{3} \ \ O_{3 \times 9} ]\tr $, while the other point $ Y $ is placed at a distance $ L^{*} = 0.95\,\pi $ from $ X $. By using our numerical algorithms, we want to recover this distance. This choice is made to have two points that are far enough from each other; i.e., this problem is such that it cannot be solved by using single shooting alone. Recall from the discussion in section~\ref{sec:connection} that a lower bound on the global injectivity radius of $ \Stnp $ is given by $ 0.8944 \, \pi $, so it makes sense to consider a distance $ L^{*} > 0.8944\,\pi $ to test the LFMS algorithm.
As number of points, we choose $ m = 4 $, i.e., the path between $ X $ and $ Y $ is partitioned into three subintervals.

To monitor the convergence behavior, two quantities have been considered:
\begin{itemize}
   \item $ | L_{k} - L^{*} | $, where $ L_{k} $ is the length of the piecewise geodesic at iteration $ k $.
   \item $ \| F(\Sigma_{k}) \|_{2} $, where $ F(\Sigma_{k}) $ is the nonlinear function of multiple shooting, as defined in~\eqref{eq:MS_nonlinear_system}.
\end{itemize}

Panel (a) of Figure~\ref{fig:convergence_LFMS} reports on the convergence behavior of leapfrog. Leapfrog is stopped when $ \| F(\Sigma_{k}) \|_{2} $ reaches the threshold value of $ 10^{-3} $ (this happens at the 28th iteration). We estimate that at this threshold, the iterates will fall in the so-called basin of attraction of Newton's method so that multiple shooting will succeed when started with the iterate generated by leapfrog. It is evident the linear convergence behavior of leapfrog.

Panel (b) of Figure~\ref{fig:convergence_LFMS} reports on the convergence behavior of multiple shooting. Multiple shooting is started from where the leapfrog algorithm left the job; one can check this by comparing the values of the monitored quantities at the last iteration in the leapfrog algorithm with those at the initial iteration of multiple shooting. We observe the quadratic convergence behavior and the onset of the plateau at around machine precision $ \epsmach \approx 10^{-16} $.

\begin{figure}[htbp]
    \centering
    \begin{minipage}{0.485\textwidth}
        \centering
        \includegraphics[width=\textwidth]{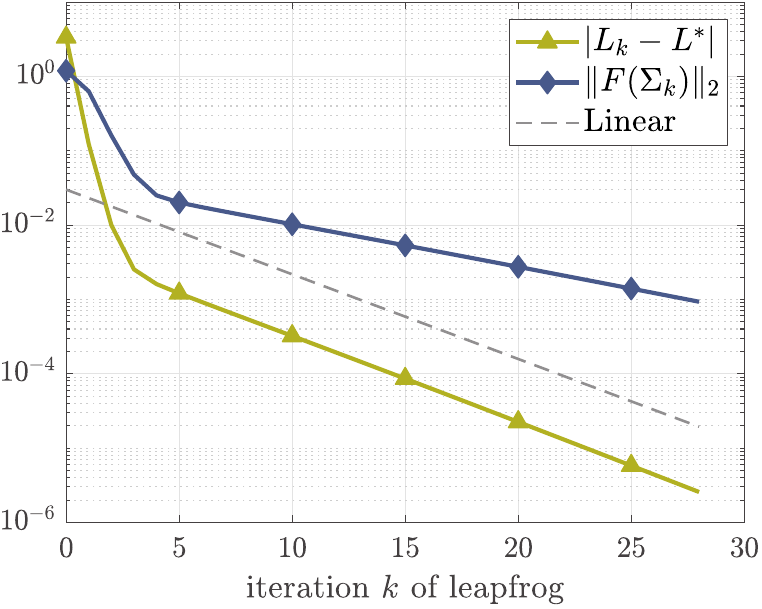}
        {\scriptsize (a)}
    \end{minipage}\hfill
    \begin{minipage}{0.502\textwidth}
        \centering
        \includegraphics[width=\textwidth]{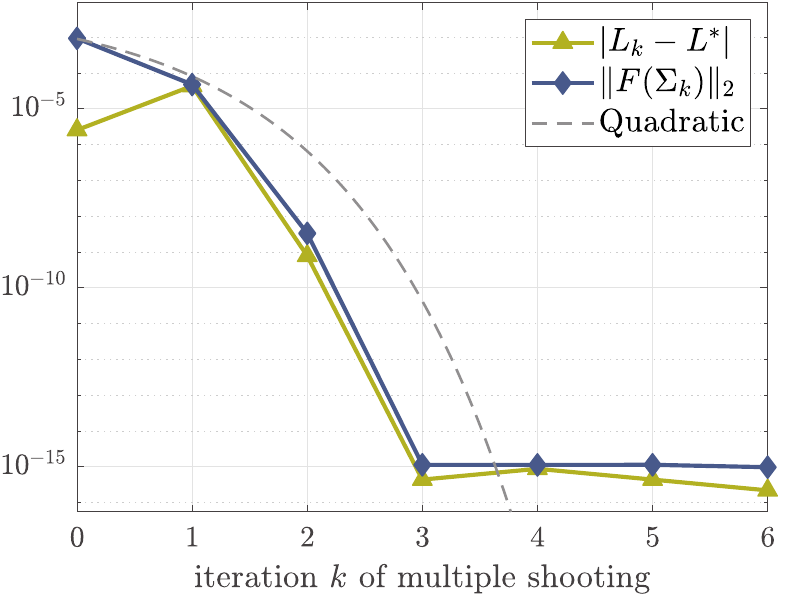}
        {\scriptsize (b)}
    \end{minipage}
    \caption{Convergence of LFMS for $ \mathrm{St}(12,3) $, with $ m = 4 $. Panel (a): convergence of leapfrog. Panel (b): convergence of multiple shooting.}\label{fig:convergence_LFMS}
\end{figure}

\section{Applications}\label{sec:applications}

In this section, we present three applications in which the calculation of the distance between points on the Stiefel manifold is involved.
We first present an application that uses means on a Riemannian manifold $ \cM $ to average probability density functions. Then, we see an application in the context of the analysis of planar shapes. Finally, we outline an application in the framework of model order reduction.

\subsection{Riemannian center of mass on the space of univariate probability density functions}\label{sec:RCM_PDFs}

We present an application that uses means on a Riemannian manifold $ \cM $. 
Given $ N $ points $ q_{i} \in \cM $, their \emph{Riemannian center of mass} is defined by the optimization problem
\begin{equation*}
    \mu = \argmin_{p\in\cM} \frac{1}{2N} \sum_{i=1}^{N} d(p,q_{i})^{2},
\end{equation*}
where $d(p,q_{i})$ is the Riemannian distance between two points on $\cM$. Seminal works on the center of mass in the context of Riemannian geometry are due to Cartan in the 1920s, Calabi in the 1950s~\protect{\cite{Afsari:2011}}, and Grove and Karcher~\protect{\cite{Grove:1973}}.

On manifolds of positive curvature, the Riemannian center of mass is generally not unique. However, if the data points $q_{i}$ are close enough to each other, then their Riemannian center of mass is unique. For an excellent study of the uniqueness of the Riemannian center of mass, we refer the reader to~\protect{\cite{Afsari:2011}}.

Here, we use the Riemannian center of mass to calculate an average probability density function (PDF). This is a simple problem since we consider the unit $ n $-sphere $ S^{n} $, a special case of the Stiefel manifold for which we also have an explicit formula for the Riemannian logarithm. However, it remains interesting because it allows us to test our algorithm and for a visualization of the outcome. It also prepares us for the application presented in the next section~\ref{sec:application_preshape_space}. We first introduce some essential notions.

Let $ \cP $ be the space of univariate PDFs on the unit interval $[0,1]$ 
\[
   \cP = \Big\{ g \colon [0,1] \to \R_{\geqslant 0} \colon \int_{0}^{1}g(x)\dx = 1 \Big\}.
\]
By introducing the \emph{half-density representation} of the elements of $ \cP $
\[
   q(t) = \sqrt{g(t)},
\]
the set $ \cP $ can be identified with the space
\[
   \cQ = \big\lbrace q \colon [ 0, 1 ] \to \R_{\geqslant 0} \colon \| q \| = 1 \big\rbrace.
\]
This identification allows us to attach a spherical structure to $ \cP $, and the unit $n$-sphere $S^{n} = \{ x \in \R^{n+1}: \|x \|=1 \}$ can be used to approximate the space of univariate PDFs on the unit interval $ [0,1] $. We refer the reader to \protect{\cite[\S 7.5.3]{Srivastava:2016}} for further details.

Given a certain number of PDFs, one might be interested in computing summary statistics of all of them, which can be given by their Riemannian center of mass.
As a concrete example, we consider three PDFs, sampled at 100 points. This discretization makes them belong to $\mathrm{St}(100,1)$, i.e., the unit sphere $S^{99}$.
Figure~\ref{fig:Riemannian_CoM_PDFs} shows the three PDFs on the left panel, and their Riemannian center of mass on the right panel. The resulting Riemannian center of mass is a PDF that summarizes the three original PDFs' features (e.g., peak locations, spread around the peaks).

\begin{figure}[htbp]
   \centering
   \includegraphics[width=0.90\columnwidth]{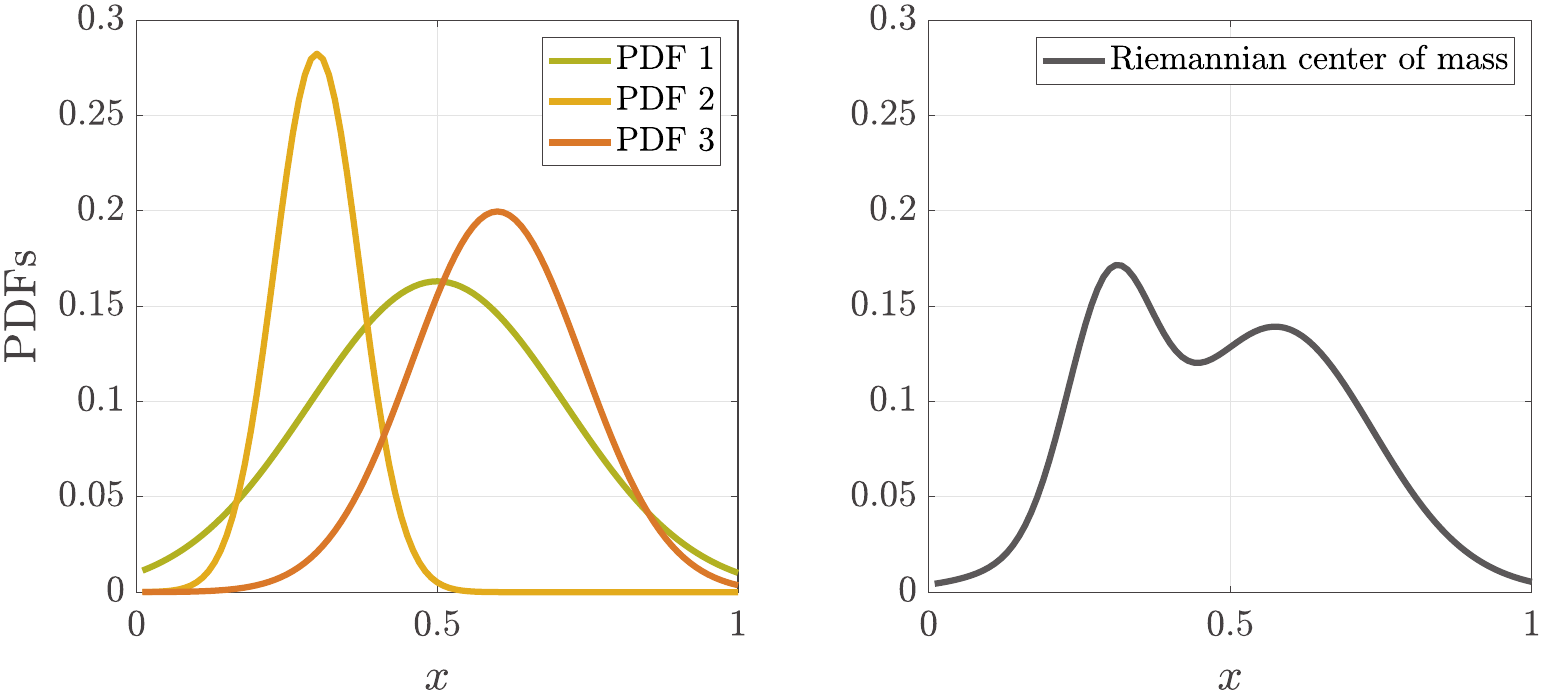}
       \begin{minipage}{0.49\textwidth}
          \centering
          \hspace{1cm} (a)
    \end{minipage}
    \begin{minipage}{0.49\textwidth}
       \centering
          (b)
    \end{minipage}
   \caption{Calculation of the Riemannian center of mass of three PDFs. Panel (a): The three PDFs. Panel (b): their Riemannian center of mass.}\label{fig:Riemannian_CoM_PDFs}
\end{figure}

\subsection{Analysis of planar shapes in the pre-shape space}\label{sec:application_preshape_space}

In this section, we analyze planar shapes in the so-called pre-shape space. In some practical applications, like imaging, the observed shapes are affected by transformations that are more complicated than similarity transformations (translation, rotation, and global scaling). For instance, when the image plane of a camera is not parallel to the plane containing the shape (distortion), or when a camera is used to image the same scene from different viewing angles. 
In these cases, one usually goes beyond the similarity transformations to define shape equivalences.

The application presented in this section was featured in~\protect{\cite[\S 11.3]{Srivastava:2016}}. Similar numerical experiments also appeared in~\protect{\cite{Younes:2008}}, where the space of planar closed curves was identified with a Grassmann manifold, using the complex square-root representation. In~\protect{\cite{Bryner:2017}}, it was done on the Stiefel manifold endowed with the embedded metric. In contrast, here we use the Stiefel manifold endowed with the canonical metric. 

We first give some mathematical definitions and then turn to the numerical experiments.
Let $ \Rnp $ be the space of point sets of size $ n $ in $ \R^{p} $, i.e., $ X \in \left[ x_{1}, \dots, x_{n} \right]\tr \in \Rnp $, and let the affine group $ G_{a} = GL(p) \ltimes \R^{p} $, $ GL(p) $ denotes the space of invertible $p$-by-$p$ matrices.
The action of the affine group $ G_{a} $ on $ \Rnp $ defines the orbits
\[
   \left[ X \right] = \left\lbrace XA + B \mid A \in GL(p), \, B = \mathbf{1}\diag(b) \right\rbrace,
\]
where $ b \in \R^{p} $, and $ \mathbf{1} $ denotes a matrix of ones of size $n$-by-$p$. We define the centroid and covariance matrix as
\[
   C_{X} \coloneqq \frac{1}{n} \sum_{i=1}^{n} x_{i}, \qquad \Sigma_{X} \coloneqq \left( X - \mathbf{1} \, \diag(C_{X}) \right)\tr \left( X - \mathbf{1} \, \diag(C_{X}) \right).
\]
For any full-rank matrix $ X $, there exists an affine-standardized point set $ X_{0} \in \left[ X \right] $ that satisfies both $ C_{X} = 0 $ (centroid at the origin) and $ \Sigma_{X} = I $ (covariance condition). That is, $X_{0}$ is an element of the Stiefel manifold $ \Stnp $. Affine standardization is the task of finding a canonical element in the affine orbit of any given curve. Roughly speaking, one can think of this operation as a projection which is typically achieved through an iterative algorithm. Since it is not our aim here to describe how to perform the affine standardization of curves, we refer the interested reader to~\protect{\cite[\S 11.3.1]{Srivastava:2016}} for more details. 

For any two affine-standardized point sets $ X_{0}^{(1)} $, $ X_{0}^{(2)} \in \left[ X \right] $, we have the equivalence relationship $ X_{0}^{(1)} \sim X_{0}^{(2)} $ up to an orthogonal transformation in $ \Op $. 
The space of all affine-standardized point sets (called affine-invariant pre-shape space)
\[
   \cA_{n,p} = \left\lbrace X \in \Rnp \mid C_{X} = 0, \, \Sigma_{X} = I \right\rbrace.
\]
The affine-invariant shape space is the quotient $ \cA_{n,p}/\Op $. We emphasize that the shape space itself does not have a manifold structure, even though the pre-shape space is a manifold \protect{\cite[\S 9.3.2]{Srivastava:2016}}.
Hence, an analysis on $ \Stnp $ alone is equivalent to an analysis on the pre-shape space $ \cA_{n,p} $, so it is \emph{not} an affine-invariant shape analysis. The references \protect{\cite{Kendall:1984}} and \protect{\cite[\S 6]{Kendall:1999}} provide more information to the interested reader.

Table~\ref{tab:geodesic_MPEG-7} illustrates geodesics on the Stiefel manifold using six shapes from the MPEG-7 dataset. For each shape group chosen, we picked two shapes as endpoints, computed the Riemannian logarithm with the shooting method, and then calculated eight equidistant intermediate points (shapes) located on geodesic joining those given endpoints. Each row in Table~\ref{tab:geodesic_MPEG-7} represents a geodesic on $ \mathrm{St}(n,2) $ ($n$ is different for every row), running from the left to the right. The two red shapes at the extrema of each row are the endpoints of a geodesic. The right column of Table~\ref{tab:geodesic_MPEG-7} reports the distances between the endpoints for each geodesic.

\begin{table}[htbp]
   \begin{center}
      \begin{tabular}{c|c}
         \toprule
         Geodesics on $ \mathrm{St}(n,2) $  &  Distance \\
         \midrule
         \includegraphics[align=c,width=0.85\textwidth]{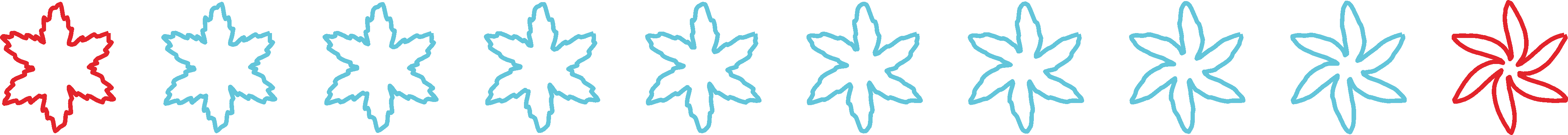}    &  0.28 \\
         \midrule
         \includegraphics[align=c,width=0.85\textwidth]{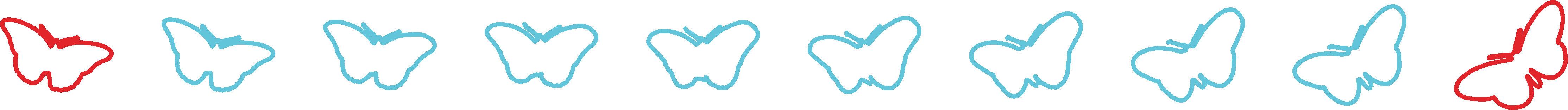}  &  1.23 \\
         \midrule
         \includegraphics[align=c,width=0.85\textwidth]{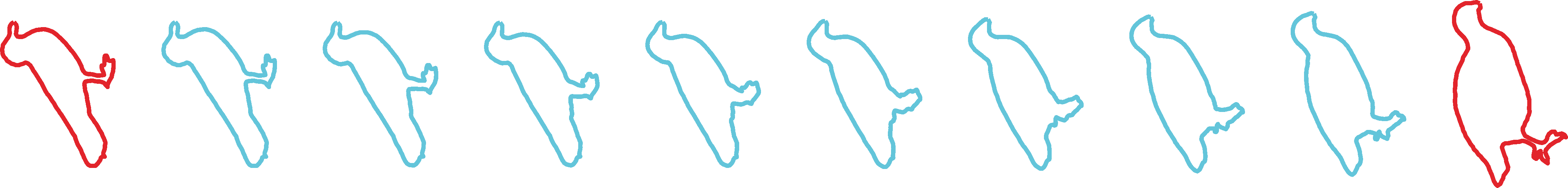}       &  0.55 \\
         \midrule
         \includegraphics[align=c,width=0.85\textwidth]{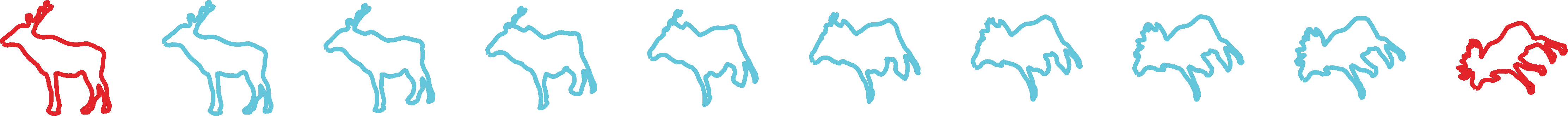}       &  0.78 \\
         \midrule
         \includegraphics[align=c,width=0.85\textwidth]{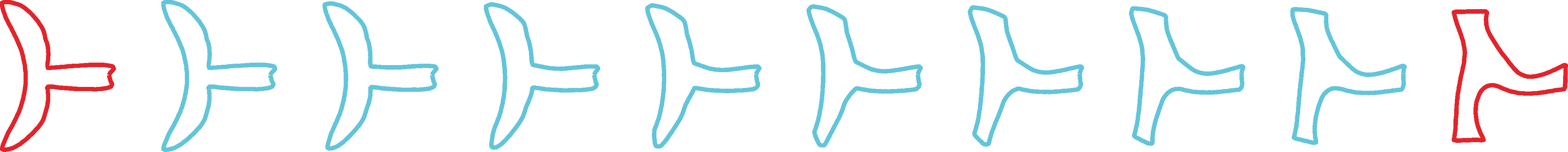}    &  0.31 \\
         \midrule
         \includegraphics[align=c,width=0.85\textwidth]{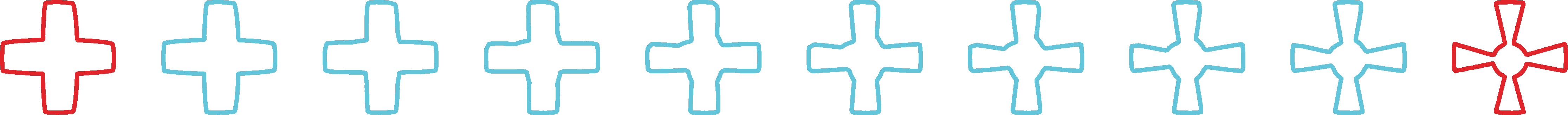}    &  0.21 \\
         \bottomrule
      \end{tabular}
   \end{center}
   \caption{Geodesics connecting MPEG-7 shapes on $ \mathrm{St}(n,2) $.}\label{tab:geodesic_MPEG-7}
\end{table}

Figure~\ref{fig:device1} illustrates an example of computing summary statistics of a group of shapes in the MPEG-7 dataset. For a given set, in this case, we choose ``device1'', whose shapes all share a certain degree of similarity, we compute its Riemannian center of mass, as defined in section~\ref{sec:RCM_PDFs}. 
Panel (a) shows all the twenty shapes in the group ``device1'', while panel (b) shows the resulting center of mass, which summarizes the features of the shapes in this group (the number of spikes, their orientation, and their thickness).

The values above each shape in panel (a) of Figure~\ref{fig:device1} are the distances of each shape from the Riemannian center of mass.
The shape closest to the Riemannian center of mass is the one in the first row, third column. We observe that the shape furthest away from the Riemannian mean is the one on the third row, third column. This might be reasonably due to the extreme thinness and shape of the spikes of this star. Conversely, the shape closest to the center of mass is the one in the second row, first column.
 
\begin{figure}[htbp]
    \centering
    \begin{minipage}{0.725\textwidth}
        \centering
        \includegraphics[width=\textwidth]{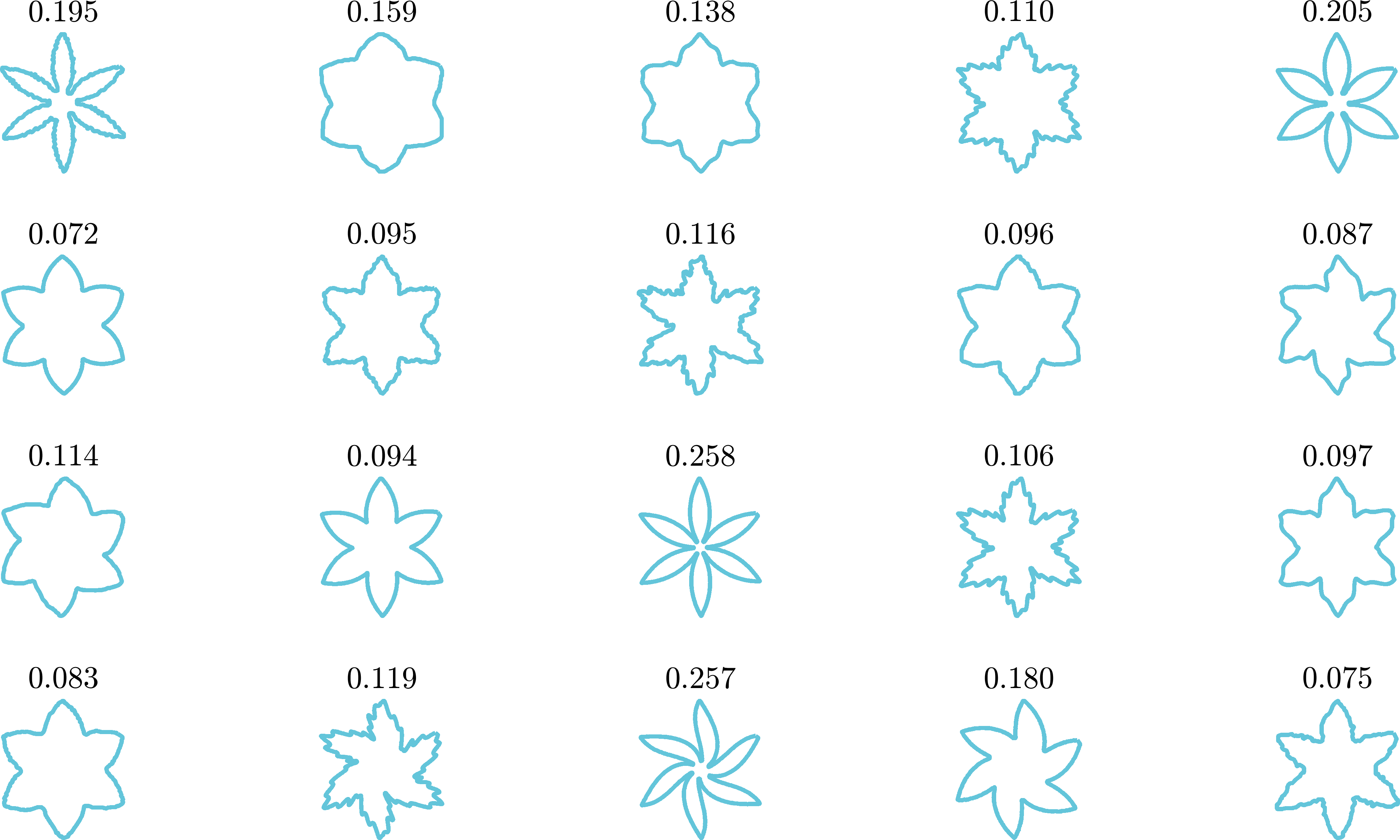}
        {\scriptsize (a)}
    \end{minipage}\hfill
    \begin{minipage}{0.24\textwidth}
        \centering
        \includegraphics[width=\textwidth]{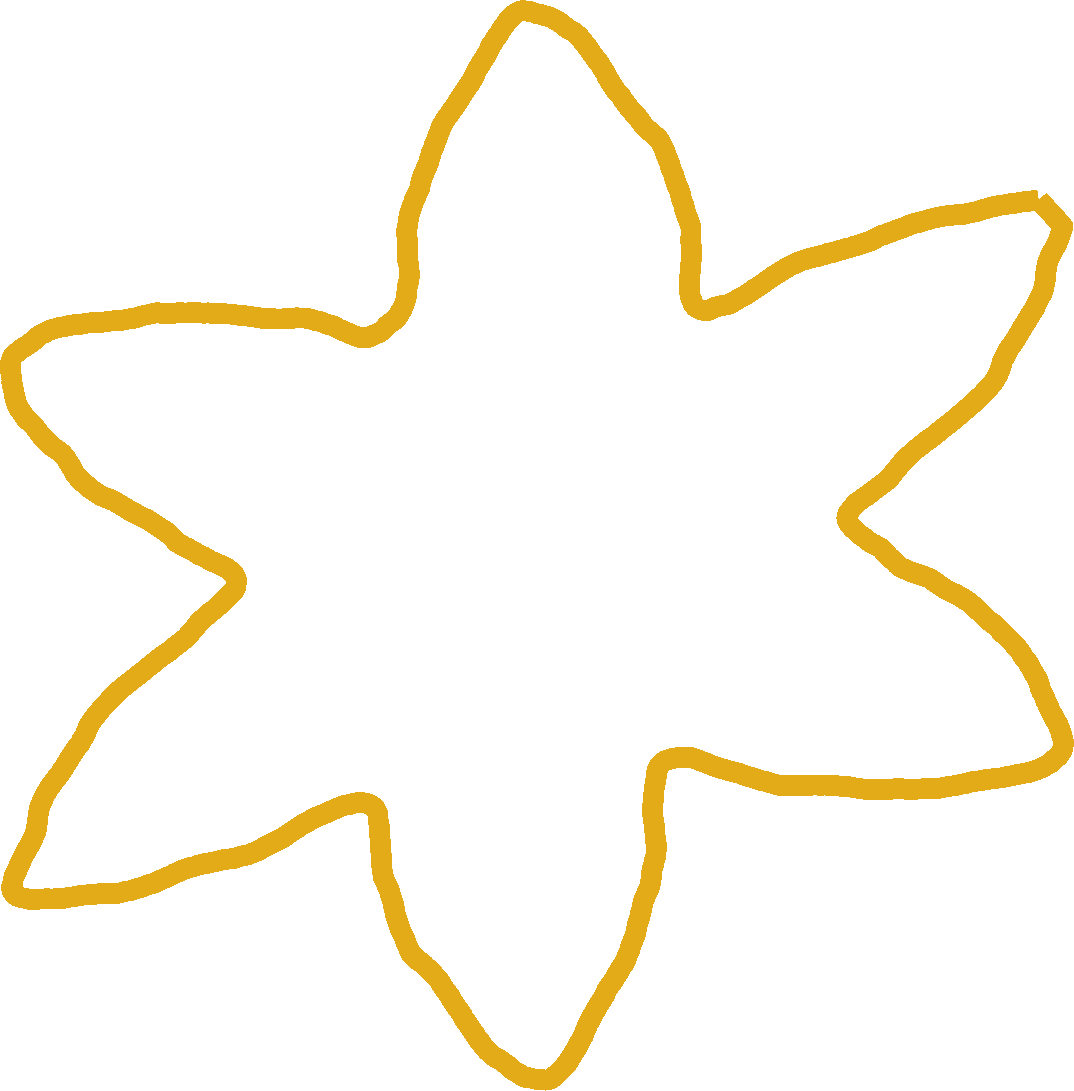}
        {\scriptsize (b)}
    \end{minipage}
    \caption{Panel (a): The shapes from the ``device1'' group in the MPEG-7 dataset. Panel (b): Their Riemannian center of mass computed with the shooting method.}\label{fig:device1}
\end{figure}

\subsection{Interpolation on the Stiefel manifold for model order reduction}\label{sec:interp_on_stiefel}

This section considers an example in the same order of ideas as in \protect{\cite{Amsallem:2011}}.
Specifically, we look at the interpolation of \emph{linear parametric reduced-order models}.
It is beyond the scope of this paper to discuss reduced-order models (ROMs); for a comprehensive review of model order reduction techniques, we refer the reader to \protect{\cite{Benner:2015}}.

Let us consider the dynamical model parameterized with respect to $ \mathbf{p} = \left[ p_1, \dots, p_d \right]\tr $
\begin{equation*}
\begin{cases}                                                                                    
      \dt{\mathbf{x}}(t;\mathbf{p}) = A(\mathbf{p})\,\mathbf{x}(t;\mathbf{p}) + B(\mathbf{p})\,\mathbf{u}(t) \\
      \mathbf{y}(t;\mathbf{p}) = C(\mathbf{p})\,\mathbf{x}(t;\mathbf{p}), \\
\end{cases}
\end{equation*}
with $ \mathbf{x}(t;\mathbf{p}) \in \R^{n} $ the vector of state variables, $ \mathbf{u}(t) \in \R^{m} $ the vector of inputs, and $ \mathbf{y}(t) \in \R^{q} $ the vector of outputs.
The system matrices are $ A(\mathbf{p}) \in \R^{n \times n}$, $ B(\mathbf{p}) \in \R^{n \times m} $, and $ C(\mathbf{p}) \in \R^{q\times n} $.

The \emph{reduced dynamical system} is
\begin{equation*}
   \begin{cases}
      \dt{\mathbf{x}}_{r}(t;\mathbf{p}) = A_{r}(\mathbf{p})\,\mathbf{x}_{r}(t;\mathbf{p}) + B_{r}(\mathbf{p})\,\mathbf{u}(t) \\
      \mathbf{y}_{r}(t;\mathbf{p}) = C_{r}(\mathbf{p})\,\mathbf{x}_{r}(t;\mathbf{p}), \\
   \end{cases}
\end{equation*}
with $ \mathbf{x}_{r} = V\tr\mathbf{x} $ the reduced-size vector, and system matrices $ A_{r} = V\tr\!A V $, $ B_{r} = V\tr\!B $, $ C_{r} = C V $, where $ V \equiv V(\mathbf{p}) \in \mathrm{St}(n,r) $. One needs to apply a ROM technique to obtain the matrix $ V $. Here, we adopt a proper orthogonal decomposition (POD) with $N$ snapshots \protect{\cite[\S 3.3.1]{Benner:2015}}.
Let $ X $ be the snapshot matrix that collects $N$ snapshots of the solution at different times $ t_{1}, \ldots, t_{N} $:
\[
   X = \big[ \mathbf{x}(t_{1};\mathbf{p}), \ldots, \mathbf{x}(t_{N};\mathbf{p}) \big].
\]
Then, the POD basis $ V $ is chosen as the $r$ left singular vectors of $ X $ corresponding to the $ r $ largest singular values. In MATLAB notation:
\[
   \big[ U, \sim, \sim \big] = \mathrm{svd}(X), \quad \text{then} \quad V = U(:,1\! :\! r).
\]

The process of interpolation on manifolds is explained in \protect{\cite[\S 4.1.2]{Amsallem:2011}} and \protect{\cite[\S 4.2.1]{Benner:2015}}. It can be summarized as follows, with Figure~\ref{fig:interpolation_stiefel} as a reference illustration.
For each parameter in a set of parameter values $\{ \mathbf{p}_1, \dots, \mathbf{p}_{K}\}$, one uses a model order reduction technique to derive a reduced-order basis $V_{i} \in \mathrm{St}(n,r)$.
This yields a set of local basis matrices $\{ V_1, \dots, V_{K}\}$. One of these matrices ($V_{3}$ in the figure) is chosen as a reference point to expand a tangent space to $ \mathrm{St}(n,r) $.
Then, given a new parameter value $\hat{\mathbf{p}}$, a basis $ \widehat{V} $ can be obtained by interpolating the local basis matrices on the tangent space. This process remains the same for general manifolds.

\begin{figure}[htbp]
   \centering
   \includegraphics[width=0.75\columnwidth]{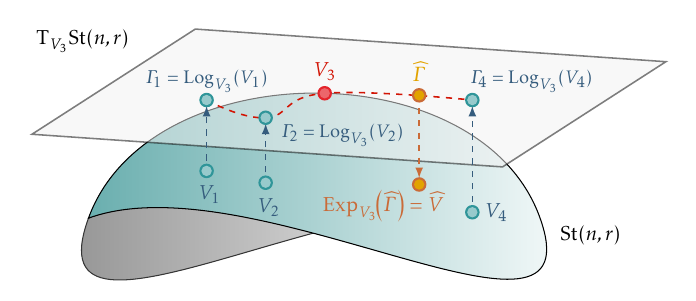}
   \caption{Interpolation on $\mathrm{St}(n,r)$.}
   \label{fig:interpolation_stiefel}
\end{figure}

As a concrete application, we consider the transient heat equation on a square domain with four disjoint discs, which model four cookies lying on a square tray in an oven \protect{\cite[\S 4.3.2]{Tobler:2012}}.
The problem is discretized with a finite element mesh with piecewise linear basis functions, resulting in a parameterized dynamical system of size $ n = 1169 $ of the form
\[
     \dt{\mathbf{x}}(t;\mathbf{p}) = - A(\mathbf{p})\, \mathbf{x}(t;\mathbf{p}) + \mathbf{b},
\]
where
\[
   \mathbf{p} = ( p_{1}, \, p_{2}, \, p_{3}, \, p_{4} ) \in [0,1]^{4}, \qquad A(\mathbf{p}) = \Big( A_{0} + \sum_{i = 1}^{4} p_{i} A_{i} \Big),
\]
and the matrices $ A_{1}, \ldots, A_{4} $ contain the contributions from the corresponding disc. The right-hand side $ \mathbf{b} $ is obtained from discretizing the source term $ f \equiv 1 $.

The simulation runs for $ t \in [0,500] $, with a time step $ \Delta t = 0.1 $.

In our example, $ \mathbf{p} = ( p_{1}, \, 0.10, \, 0.15, \, 0.70 ) $, with $ p_{1} \in [ 0.12, 1 ] $, i.e., the first parameter varies while the others are fixed.
As a ROM technique, we adopt a POD with 500 snapshots in time, with a reduced-model size $ r = 4 $.

We monitored the following error quantities:
\begin{itemize}
\item The error between $ \widehat{V}_{\mathrm{POD}} $, the basis obtained by directly applying a POD, and $ \widehat{V}_{\mathrm{interp}} $, the basis obtained by interpolating on $ \mathrm{St}(n,r) $ as described above:
\[
    \textrm{err-interp} = \| \widehat{V}_{\mathrm{POD}} - \widehat{V}_{\mathrm{interp}}\|_{2}.
\]
\item For the new operating point $ \mathbf{\hat{p}} = ( \hat{p}_{1}, \, 0.10, \, 0.15, \, 0.70 ) $, with $ \hat{p}_{1} = 0.40 $, the relative error on the output of the reduced model with respect to the output of the full model (see \protect{\cite[\S 2.4]{Benner:2015}}):
\[
    \textrm{err-$\mathbf{y}$} = \frac{\| \mathbf{y}_{r}(t,\mathbf{\hat{p}}) - \mathbf{y}(t,\mathbf{\hat{p}}) \|_{L_2}}{\|\mathbf{y}(t,\mathbf{\hat{p}}) \|_{L_2}}.
\]
\end{itemize}

To perform the interpolation on the tangent space, the MATLAB function \verb+interp1+ for 1D interpolation was used with three different methods: piecewise linear interpolation (linear), piecewise cubic spline interpolation (spline), and shape-preserving piecewise cubic interpolation (pchip).

Panel (a) of Figure~\ref{fig:interpolation_on_St} reports on the convergence behavior of $ \textrm{err-interp} $ with respect to the number $ K $ of local basis matrices. It is clear  that $ \textrm{err-interp} $ improves as we increase the number of local basis matrices. Moreover, the spline method appears to be the most accurate among the ones considered.

Next, we monitor the convergence behavior of $ \textrm{err-$\mathbf{y}$} $ with respect to the size $ r $ of the reduced model, $ r = 1, 2, \ldots, 20 $.
We choose $ \mathbf{p} = (0.12, \, 0.10, \, 0.15, \, 0.70) $ and considered five different PODs, with increasing snapshots, 10, 100, 500, 1\,000, 2\,500 respectively. We estimate that an $ \textrm{err-$\mathbf{y}$} $ of about $ 1\% $ is already good enough for applications in various engineering fields.
From Panel (b) of Figure~\ref{fig:interpolation_on_St}, one can observe that for reduced models obtained from 500, 1\,000, 2\,500 snapshot PODs, the $ 1\% $ error is achieved for a size $ r=4 $. On the other hand, when using fewer snapshots (like 10, 100), one needs $ r = 9,\,10 $ to achieve $ \textrm{err-$\mathbf{y}$} = 1\% $.

\begin{figure}[htbp]
    \centering
    \begin{minipage}{0.485\textwidth}
        \centering
        \includegraphics[width=\textwidth]{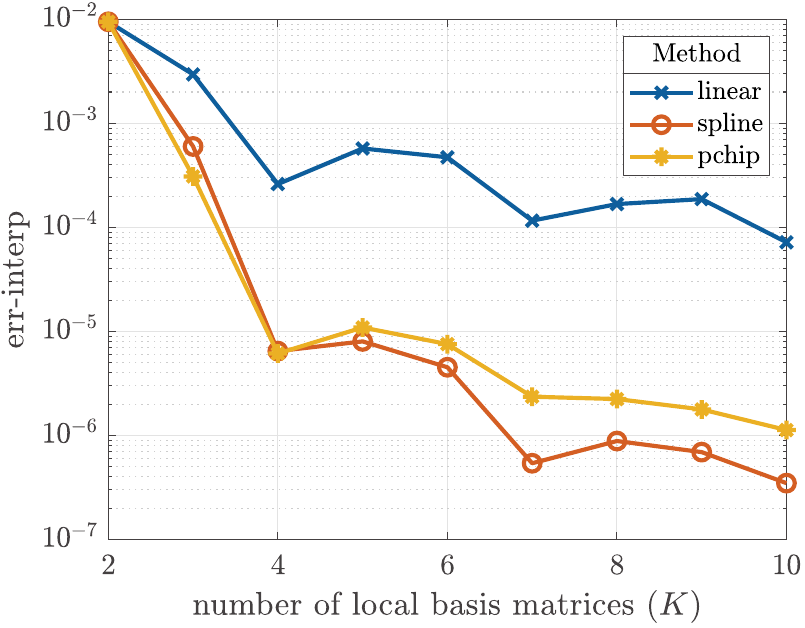}
        {\scriptsize (a)}
    \end{minipage}\hfill
    \begin{minipage}{0.485\textwidth}
        \centering
        \includegraphics[width=\textwidth]{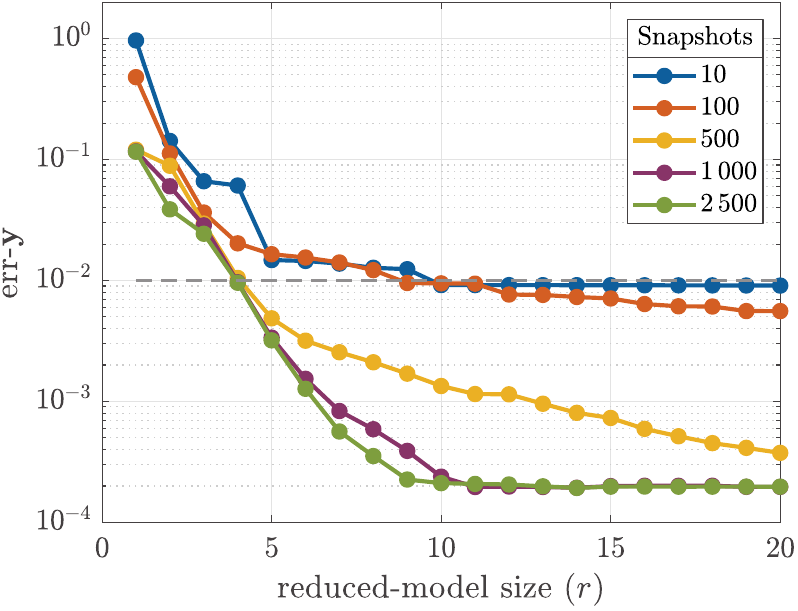}
        {\scriptsize (b)}
    \end{minipage}
    \caption{Panel (a): error of the interpolation on the Stiefel manifold. Panel (b): error of the ROMs with respect to the full model.}\label{fig:interpolation_on_St}
\end{figure}

\section{Conclusions and outlook} \label{sec:conclusions}

In this work, we studied the shooting method, a classical numerical algorithm for solving boundary value problems, to compute the distance between two given points on the Stiefel manifold. We provided shooting methods for calculating geodesics on the Stiefel manifold with neat formulas for the Jacobians involved. We offered an elegant way to start the algorithm and obtain the desired quadratic convergence.
Moreover, we conducted a preliminary analysis of the explicit expression for the Jacobian of the matrix exponential involved in the single shooting method and related it to differential geometric properties.
Numerical experiments demonstrate the algorithms in terms of performance and accuracy, while the applications considered show how they can be used in practical circumstances.

As a future outlook, we may use the knowledge gained in this work to develop a computationally cheaper algorithm. For example, when $ \|A\|_{2} $ is relatively small, we may approximate the Jacobian of the matrix exponential by the identity matrix. We expect that we will lose in the order of convergence of the method, but we will gain in terms of time and memory storage. 
Another promising research direction is exploring the connection between shooting algorithms for geodesics and domain decomposition methods. These topics will be the object of future studies.

\section*{Acknowledgments}

The author is grateful to Bart Vandereycken for his guidance during the author's Ph.D. thesis.
Part of this work was completed during the author's Ph.D. thesis at the University of Geneva, SNSF fund number 163212\footnote{SNSF webpage: \href{https://data.snf.ch/grants/grant/163212}{https://data.snf.ch/grants/grant/163212}}. It was completed during the author's postdoctoral fellowship at the National Center for Theoretical Sciences in Taiwan.

\appendix

\section{Single shooting} \label{app:simple-shooting}

\subsection{Freedom in choosing the geodesic}\label{sec:freedom}

As mentioned in Remark~\ref{rmk:Y_0perp}, the matrix $Y_{0\perp}$ does not need to be orthonormal; in fact, its only requirement is that it has to span $\cY_{0}^{\perp}=(\mathrm{span}(Y_{0}))^{\perp}$, the orthogonal subspace to $\cY_{0}=\mathrm{span}(Y_{0})$. In this appendix, we are going to show this, starting from the geodesic
\[
Y(t) = \big[ Y_{0} \ \  Y_{0\perp} \big] \, \exp\!\left( \begin{bmatrix}
\Omega   &   -K\tr \\
K        &    O_{n-p}
\end{bmatrix} t \right)
\begin{bmatrix}
I_{p} \\
O_{(n-p)\times p}
\end{bmatrix}.
\]
Let $ M $ be any $ (n-p) $-by-$(n-p)$ invertible matrix, and define
\[
\widetilde{M} =
\begin{bmatrix}
I_{p} &   \\
      & M
\end{bmatrix}, \qquad
\widetilde{M}^{-1} =
\begin{bmatrix}
I_{p} &   \\
      & M^{-1}
\end{bmatrix},
\qquad
\widetilde{M}^{-1} \widetilde{M} = \widetilde{M}\widetilde{M}^{-1} = I_{n}.
\]
Observe that
\[
\big[ Y_{0} \ \  Y_{0\perp} \big] \, \widetilde{M}^{-1} = \big[ Y_{0} \ \  Y_{0\perp} \big]  \begin{bmatrix}
I_{p} &    \\
      & M^{-1}
\end{bmatrix} = \big[ Y_{0} \ \ Y_{0\perp} M^{-1} \big],
\]
and
\[
\widetilde{M} \begin{bmatrix}
I_{p} \\
O_{(n-p)\times p}
\end{bmatrix} =
\begin{bmatrix}
I_{p} &    \\
      & M
\end{bmatrix}
\begin{bmatrix}
I_{p} \\
O_{(n-p)\times p}
\end{bmatrix}=
\begin{bmatrix}
I_{p} \\
O_{(n-p)\times p}
\end{bmatrix}.
\]
In the following steps, we use these facts together with the property 
\[
   \widetilde{M} \exp(A) \widetilde{M}^{-1} = \exp(\widetilde{M} A \widetilde{M}^{-1}),
\]   
which holds for any invertible matrix $\widetilde{M}$.
\begin{align*}
Y(t) & = \big[ Y_{0} \ \  Y_{0\perp} \big] \, \widetilde{M}^{-1} \widetilde{M} \, \exp\!\left( \begin{bmatrix}
\Omega   &   -K\tr \\
K        &    O_{n-p}
\end{bmatrix} t \right) \widetilde{M}^{-1} \widetilde{M}
\begin{bmatrix}
I_{p} \\
O_{(n-p)\times p}
\end{bmatrix}, \\
& = \big[ Y_{0} \ \ Y_{0\perp} M^{-1} \big] \, \exp\!\left( \widetilde{M} \begin{bmatrix}
\Omega   &   -K\tr \\
K        &    O_{n-p}
\end{bmatrix} \widetilde{M}^{-1} t \right)
\begin{bmatrix}
I_{p} \\
O_{(n-p)\times p}
\end{bmatrix}, \\
& =  \big[ Y_{0} \ \ Y_{0\perp} M^{-1} \big] \, \exp\!\left( \begin{bmatrix}
\Omega   &   -K\tr M^{-1} \\
M K        &    O_{n-p}
\end{bmatrix}  t \right)
\begin{bmatrix}
I_{p} \\
O_{(n-p)\times p}
\end{bmatrix}.
\end{align*}
Since the matrix $ M $ is invertible, it can be regarded as a change of basis. Hence, it appears from the last expression that there is freedom in choosing $Y_{0\perp}$ since it can be any matrix whose columns form a basis for $\cY_{0}^{\perp}$.

\subsection{Smaller formulation}\label{sec:baby}

In this section, we prove that when $ p \leqslant \frac{n}{2} $, the geodesic problem on $\Stnp$ can be reformulated into an equivalent problem on $\mathrm{St}(2p,p)$.
We start from \eqref{eq:closed-form-sol-geodesic} with $t=1$, namely,
\[
Y_{1} = \big[ Y_{0} \ \  Y_{0\perp} \big] \exp\!\left( \begin{bmatrix}
\Omega   &   -K\tr \\
K        &    O_{n-p}
\end{bmatrix} \right)
\begin{bmatrix}
I_{p} \\
O_{(n-p)\times p}
\end{bmatrix}.
\]
Consider the QR decomposition of $K$
\[
K = \big[ Q \ \ Q_{\perp} \big] \begin{bmatrix}
R \\ O_{(n-2p)\times p}
\end{bmatrix} = Q R,
\]
where $\big[ Q \ \ Q_{\perp} \big] \in \R^{(n-p)\times (n-p)} $ is the orthogonal factor of $K$, whose blocks $Q \in \R^{(n-p)\times p}$, $Q_{\perp} \in \R^{(n-p)\times (n-2p)}$ are orthonormal, and $R\in\R^{p\times p}$ is upper triangular.
Inserting this decomposition into the matrix
\[
\begin{bmatrix}
\Omega   &   -K\tr \\
K        &    O_{n-p}
\end{bmatrix},
\]
we get
\begin{small}
\[
\begin{bmatrix}
\Omega  &  \big[ -R\tr  \ \  O \big] \big[ Q \ \ Q_{\perp} \big]\tr \\
\big[ Q \ \ Q_{\perp} \big] \! \begin{bmatrix}
R \\ O
\end{bmatrix}  &  O_{n-p}
\end{bmatrix}
=
\begin{bmatrix}
I_p  &   \\
     &  \big[ Q \ \ Q_{\perp} \big]
\end{bmatrix}
\begin{bmatrix}
\Omega  &  -R\tr   &   \\
R       &     O_p  &  \\
        &          &  O_{n-2p}
\end{bmatrix}
\begin{bmatrix}
I_p  &  \\
     &  \big[ Q \ \ Q_{\perp} \big]\tr
\end{bmatrix}.
\]
\end{small}
Substituting this expression into the argument of the matrix exponential, and using the property $ \exp( \widetilde{Q} M \widetilde{Q}\tr ) = \widetilde{Q}\exp(M) \widetilde{Q}\tr $ for any orthogonal matrix $ \widetilde{Q} $,  we get
\[
Y_{1} = \big[ Y_{0} \ \  Y_{0\perp} \big]
\begin{bmatrix}
I_p   &  \\
      &  \big[ Q \ \ Q_{\perp} \big]
\end{bmatrix}
\exp\!\left( 
\begin{bmatrix}
\Omega  &  -R\tr   &  \\
R       &     O_p  &  \\
        &          &  O_{n-2p}
\end{bmatrix}
\right)
\begin{bmatrix}
I_p   &    \\
      &  \big[ Q \ \ Q_{\perp} \big]\tr
\end{bmatrix}
\begin{bmatrix}
I_{p} \\
O_{(n-p)\times p}
\end{bmatrix}.
\]
Using the fact that the argument of $ \exp $ is a block diagonal matrix, we can write
\[
Y_{1} = \begin{bmatrix} Y_{0}  & Y_{0\perp} Q  & Y_{0\perp}Q_{\perp} \end{bmatrix}
\begin{bmatrix}
\exp\!\left( \begin{bmatrix} \Omega  &  -R\tr \\  R       &     O_p \end{bmatrix} \right)   &  \\
 &  I_{(n-2p)}
\end{bmatrix}
\begin{bmatrix}
I_{p} \\
O_{p} \\
O_{(n-2p)\times p}
\end{bmatrix}.
\]
We collect the matrices to make the products conformable
\[
Y_{1} = \begin{bmatrix} \underbrace{\big[ Y_{0} \ \ Y_{0\perp} Q \big]}_{\in\R^{n\times 2p}} & \underbrace{Y_{0\perp}Q_{\perp}}_{\in\R^{n\times (n-2p)}} \end{bmatrix}
\begin{bmatrix}
\exp\!\left( \begin{bmatrix} \Omega  &  -R\tr \\  R       &     O_p \end{bmatrix} \right)   &   \\
 &  I_{(n-2p)}
\end{bmatrix}
\begin{bmatrix}
\begin{bmatrix}
I_{p} \\
O_{p}
\end{bmatrix}\\
O_{(n-2p)\times p}
\end{bmatrix}.
\]
Finally, we have obtained the smaller formulation \eqref{eq:baby}
\begin{equation*}
Y_{1} = \big[ Y_{0} \ \ Y_{0\perp} Q \big] \, \exp\!\left( \begin{bmatrix} \Omega  &  -R\tr \\  R       &     O_p \end{bmatrix} \right)\begin{bmatrix}
I_{p} \\
O_{p}
\end{bmatrix}.
\end{equation*}

\section{Fr\'{e}chet derivatives} \label{app:frechet-derivatives}

\subsection{First-order Fr\'{e}chet derivative of a matrix function}
The Fr\'{e}chet derivative of a matrix function $f \colon \C^{n\times n} \to \C^{n\times n}$ at $X \in \C^{n\times n}$ is the unique linear function $ \D\! f(X) [\cdot] $ of the matrix $E \in \C^{n\times n}$, that satisfies
\begin{equation}\label{eq:1}
   f(X+E) - f(X) - \D\! f(X)[E] = o(\| E \|).
\end{equation}
The mapping itself is denoted by either $ \D\! f(X)[\cdot] $ or $ \D\! f(X) $, while the value of the mapping for direction $E$ (i.e., the directional derivative) is denoted by $\D\! f(X)[E]$.

Since $ \D\! f(X) \colon \C^{n\times n} \to \C^{n\times n}$ is a linear operator, one can write
\begin{equation}\label{eq:2}
   \vecop\!\big( \D\! f(X)[E] \big) = J_{f}^{X} \vecop(E)
\end{equation}
for some $n^2\times n^2$ complex matrix $ J_{f}^{X} $ independent of $E$. We refer to $ J_{f}^{X} $ as the Kronecker representation of the Fr\'{e}chet derivative, or simply as the Jacobian matrix.

\subsection{Singular values of \texorpdfstring{$ J_{f}^{X} $}{TEXT}}\label{app:sing_val_KX}

In this section, we report some results that are used in the analysis of the Jacobian $ J_{\exp(A)}^{A} $ (see section~\ref{sec:analysis_J_exp}).
The operator norm of $\D\! f(X)$ for the Frobenius norm is defined by
\[
   \| \D\! f(X) \|_{\F} = \max_{Z \neq 0} \frac{\| \D\! f(X)[Z] \|_{\F}}{\| Z \|_{\F}} = \max_{\|Z\|_{\F}=1} \| \D\! f(X)[Z] \|_{\F}.
\]
By vectorizing $ \D\! f(X)[Z] $ as in~\eqref{eq:2}, and using the fact that the Euclidean norm of $z = \vecop(Z)$ equals the Frobenius norm of $Z$, we can also write
\[
   \| \D\! f(X) \|_{\F} = \max_{\| z \|_{2} = 1} \|  J_{f}^{X}  \, z \|_{2}  = \|  J_{f}^{X}  \|_{2} = \sigma_{\max}( J_{f}^{X} ),
\]
where $ \sigma_{\max}( J_{f}^{X} ) $ is the largest singular value of $ J_{f}^{X} $.

We have the following important theorem.

\begin{theorem}[\protect{\cite[Cor.~3.16]{Higham:2008}}]\label{thm:max_sing_K}
If $X \in \mathbb{C}^{n\times n}$ is normal, then
\begin{equation}\label{eq:max_sing_K}
   \sigma_{\max}( J_{f}^{X} ) = \max_{\lambda,\mu \in \Lambda(X)} \big| f[\lambda,\mu] \big|,
\end{equation}
where $ \Lambda(X) $ denote the eigenvalues of $X$, and $f[\lambda,\mu]$ is the first-order divided difference defined by 
\begin{equation}\label{eq:def_div_diff}
f[\lambda,\mu] = \begin{cases}
\ \frac{f(\lambda)-f(\mu)}{\lambda-\mu}, & \lambda\neq \mu, \\
\ f'(\lambda), & \lambda = \mu.
\end{cases}
\end{equation}
\end{theorem}

If $ \D\! f(X) $ is invertible, we have a similar property for the minimal singular value:
\begin{theorem}\label{thm:min_sing_K}
If $X \in \mathbb{C}^{n\times n}$ is normal, then
\[
   \sigma_{\min}( J_{f}^{X} ) = \min_{\lambda,\mu \in \Lambda(X)} \big| f[\lambda,\mu] \big|.
\]
\end{theorem}
\begin{proof}
We adjust the proof of \protect{\cite[Cor.~3.16]{Higham:2008}} accordingly.
We start from the variational property \protect{\cite[Theorem~8.6.1]{Golub:2013}}
\[
   \sigma_{\min}( J_{f}^{X} ) =  \min_{\|E\|_{\F}=1} \| \D\! f(X)[E] \|_{\F},
\]
and we use $ \D\! f(X)[E] = Z \big( \D\! f(D)[\widetilde{E}] \big) Z^{-1} $, with $ D = \diag(\lambda_{i}) $ and $ \widetilde{E} = Z^{-1} E Z $, as in \protect{\cite[Cor.~3.12]{Higham:2008}}. Then
\[
   \sigma_{\min}( J_{f}^{X} ) =  \min_{\|\widetilde{E}\|_{\F}=1} \| Z \big( \D\! f(D)[\widetilde{E}]\big) Z^{-1} \|_{\F} = \min_{\|\widetilde{E}\|_{\F}=1} \| \D\! f(D)[\widetilde{E}] \|_{\F} = \min_{i,j} \big| f[\lambda_{i},\lambda_{j}] \big|,
\]
where for the last equality we used the same reasoning as in the proof of \protect{\cite[Cor.~3.13]{Higham:2008}}.
\end{proof}

\subsection{Proof of Lemma~\ref{lemma:sing_val_J_expA}}\label{app:proof_lemma_sing_val_J_expA}

\begin{proof}
Since $A$ is a real skew-symmetric matrix, the eigenvalues of $A$ are purely imaginary. Hence, we may denote them as $ \icomp\! x$ and $ \icomp\! y$, with $ x, y \in \R $.
Let us rewrite \eqref{eq:max_sing_K} as
\[
   \| J_{\exp(A)}^{A} \|_{2} = \sigma_{\max}\big(J_{\exp(A)}^{A}\big) = \max_{|x|,|y| \leqslant \alpha} \big| \exp [\icomp\!x,\icomp\!y] \big|,
\]
where $|x|,|y| \leqslant \alpha $ because the absolute value of an eigenvalue of a normal matrix cannot exceed any norm of that matrix.
The maximum is attained for $ y = x $, and using the definition in \eqref{eq:def_div_diff}, we get
\[
   \sigma_{\max}\big(J_{\exp(A)}^{A}\big) = \max_{|x| \leqslant \alpha} \big| \exp [\icomp\!x,\icomp\!x] \big| = \max_{|x| \leqslant \alpha} \big| \exp'(\icomp\!x) \big| = \max_{|x| \leqslant \alpha} \big| \exp(\icomp\!x) \big| = 1.
\]
This shows that the maximum singular value of $J_{\exp(A)}^{A}$ is always 1.

For the minimum singular value, let us specialize Theorem~\ref{thm:min_sing_K} to our case:
\[
   \sigma_{\min}\big(J_{\exp(A)}^{A}\big) = \min_{|x|,|y| \leqslant \alpha} \big| \exp [\icomp\!x,\icomp\!y] \big| = \min_{|x|,|y| \leqslant \alpha } \underbrace{\left\vert \frac{e^{\icomp\!x} - e^{\icomp\!y}}{\icomp\!x-\icomp\!y} \right\vert}_{\eqqcolon g(x,y)}.
\]
The minima of $g(x,y)$ are attained on the anti-diagonal at the corners, namely, when $ x = \alpha $, $ y = -\alpha $ and $ x = -\alpha $, $ y = \alpha $. This gives:
\[
\sigma_{\min}\big(J_{\exp(A)}^{A}\big) = \left\vert \frac{e^{ \icomp\!\alpha} - e^{- \icomp\!\alpha}}{2 \icomp\!\alpha} \right\vert = \left\vert \frac{\sin \alpha}{\alpha} \right\vert = \left\vert \sinc \alpha \right\vert.
\]
\end{proof}

\section{Jacobians for multiple shooting} \label{app:multiple-shooting}

This appendix reports the explicit formulas for the Jacobian matrices used in the multiple shooting method on the Stiefel manifold $ \Stnp $ (see section~\ref{sec:multiple_shooting}).

Let $ \Sigma_{1} $ denote a base point and $ \Sigma_{2} $ its corresponding tangent vector as explained in section~\ref{sec:multiple_shooting} and illustrated in Figure~\ref{fig:multiple_shooting}.

To compute the Jacobian matrices appearing in \eqref{eq:MS_jacobian_G_k}, we formulate the geodesic equation \eqref{eq:closed-form-sol-geodesic} using the singular value decomposition of the base point $\Sigma_{1}$, namely,  $ \Sigma_{1} = USV\tr $.
Let us consider the partitioned matrices (MATLAB notation)
\begin{equation*}
U_{p} = U(\colon\!,1\colon\!p), \quad U_{\perp} = U(\colon\!,p+1 \colon\!\mathrm{end}), \quad V_{p} = V(\colon\!,1\colon\!p), \quad V_{\perp} = V(\colon\!,p+1\colon\!\mathrm{end}),
\end{equation*}
and let $ \widetilde{Q} = \big[ \Sigma_{1} \ \ U_{\perp} \big]$. Then, the SVD formulations of the geodesic and its derivative are
\begin{equation*}
    Z_1(t) = \widetilde{Q} \, \exp\!\left( t A \right)
    \begin{bmatrix}
        I_p \\
        O_{(n-p)\times p}
    \end{bmatrix},\qquad
    Z_2(t) = \widetilde{Q} \, \exp\!\left( t A \right) A
    \begin{bmatrix}
        I_p \\
        O_{(n-p)\times p}
    \end{bmatrix},
\end{equation*}
where
\begin{equation*}
    A(\widetilde{Q},\Sigma_{2}) = \begin{bmatrix}
        \big[ I_{p} \ \ O \big] \widetilde{Q}\tr \Sigma_{2}     &  -\Big[ \big[ O \ \ I_{n-p} \big] \widetilde{Q}\tr \Sigma_{2}\Big]\tr  \\[6pt]
        \big[ O \ \ I_{n-p} \big] \widetilde{Q}\tr \Sigma_{2}  &   O_{n-p}
    \end{bmatrix} =
    \begin{bmatrix}
        \Sigma_{1}\tr \Sigma_{2}     &  -( U_{\perp}\tr \Sigma_{2})\tr  \\[6pt]
        U_{\perp}\tr \Sigma_{2}  &   O_{n-p}
    \end{bmatrix}.
\end{equation*}

\subsection{Jacobians with respect to the base point}

Let us first compute the Jacobians of the geodesic and its derivative with respect to the base point $\Sigma_{1}$, i.e.,
\[
   J_{Z_1}^{\Sigma_{1}} = \frac{\partial Z_1}{\partial \Sigma_{1}} \quad \text{and} \quad J_{Z_2}^{\Sigma_{1}} = \frac{\partial Z_2}{\partial \Sigma_{1}}.
\]
We adopt for the functions involved the notation:
\begin{itemize}
   \item $s(\Sigma_{1}) = \mathtt{svd}(\Sigma_{1}) $, performs the SVD of $ \Sigma_{1} $ and returns $ U_{p} $,   $ U_{\perp} $, $ V_{p} $, $ V_{\perp} $;
   \item $\tilde{q}(s(\Sigma_{1}))= \big[ \Sigma_{1} \ \ U_{\perp} \big] = \widetilde{Q}$, builds the matrix $ \widetilde{Q} $ from $ \Sigma_{1} $ and $ U_{\perp} $;
   \item $h(\tilde{q}(s(\Sigma_{1}))) = A $, builds the matrix argument of $ \exp $;
   \item $g(h(\widetilde{Q})) = \exp(A)$, performs the matrix exponential of $ A $.
\end{itemize}
To compute $\frac{\partial Z_1}{\partial \Sigma_{1}}$ we have to consider the derivative of a product and the chain rule for a composite function:
\begin{equation*}
\begin{split}
    \D\! Z_{1}(\Sigma_{1})[E] = & \ \D\! \tilde{q}\big( s(\Sigma_{1}), \ \D\! s(\Sigma_{1})[E]\big) \, \exp(A) \begin{bmatrix} I_p \\ O \end{bmatrix} \\
    & + \widetilde{Q} \, \D\! g \Big( h(\tilde{q}(s(\Sigma_{1}))), \ \D\! h \big[ \tilde{q}(s(\Sigma_{1})), \ \D\! \tilde{q}\left( s(\Sigma_{1}), \ \D\! s(\Sigma_{1})[E] \right) \big] \Big) \begin{bmatrix} I_p \\ O \end{bmatrix}.
\end{split}
\end{equation*}
As in appendix~\ref{app:frechet-derivatives}, $ \D\! f(A)[E] $ denotes the Fr\'{e}chet derivative of $ f $ at the matrix $ A $ in the direction of $ E $.
Vectorizing the last expression we get
\begin{equation}\label{eq:vec_LZ_1}
    \vecop(\D\! Z_{1}(\Sigma_{1})[E]) = \left(  \big[ I_{p} \ \ O  \big] \exp(A)\tr \otimes I_n \right)\vecop(\D\! \tilde{q}) + \left( \big[ I_{p} \ \ O  \big] \otimes \widetilde{Q} \right) \vecop(\D\! g).
\end{equation}
Here,
\begin{equation*}
    \vecop(\D\! g(A)[E]) = J_{\exp(A)}^{A} \vecop(\D\! h),
\end{equation*}
with $ J_{\exp(A)}^{A} $ the Jacobian of $ \exp $ with respect to its argument.
As we did for single shooting (see section~\ref{sec:ss_parametrization}), we introduce a linear map $ T $ that maps a block-wise vectorization into the ordinary column-stacking vectorization. This is achieved by:
\begin{equation*}
    \vecop(\D\! h) = T \cdot \blkvec(\D\! h),
\end{equation*}
where
\begin{equation*}
    \blkvec(\D\! h) = \begin{bmatrix}
        \vecop\!\big( \big[ I_{p} \ \ O  \big]  \D\! \tilde{q}\tr \Sigma_{2} \big) \\[4pt]
        \vecop\!\big( \big[ O \ \ I_{n-p} \big]  \D\! \tilde{q}\tr \Sigma_{2} \big) \\[4pt]
       -\vecop\!\big( \big[ O \ \ I_{n-p} \big] \D\! \tilde{q}\tr \Sigma_{2} \big)\tr  \\[4pt]
        \vecop(O_{n-p})
    \end{bmatrix} = J_{h}^{\Sigma_{1}} \vecop(\D\! \tilde{q}\tr) = J_{h}^{\Sigma_{1}} \, \Pi_{n,n} \, \vecop(\D\! \tilde{q}),
\end{equation*}
with
\begin{equation*}
    J_{h}^{\Sigma_{1}} = \begin{bmatrix}
        \Sigma_{2}\tr \otimes \big[ I_{p} \ \ O  \big] \\[4pt]
        \Sigma_{2}\tr \otimes \big[ O \ \ I_{n-p} \big] \\[4pt]
        -\Pi_{(n-p),p} \, \big( \Sigma_{2}\tr \otimes \big[ O \ \ I_{n-p} \big] \big) \\[4pt]
        O_{(n-p)^2 \times n^2}
    \end{bmatrix}.
\end{equation*}
Observe that
\begin{equation*}
\vecop\!\big(\widetilde{Q}\big) = \vecop\!\big(\big[\Sigma_{1} \ \ U_{\perp}\big]\big) = \begin{bmatrix}
\vecop(\Sigma_{1}) \\[4pt]
\vecop(U_{\perp})
\end{bmatrix},
\end{equation*}
hence
\begin{equation}\label{eq:vec_Lq_tilde}
    \vecop( \D\! \tilde{q}(\Sigma_{1})[E] ) = \begin{bmatrix}
        \vecop(\D\! \Sigma_{1}) \\[4pt]
        \vecop(\D\! U_{\perp})
    \end{bmatrix} = \begin{bmatrix}
        I_{np} \\[4pt]
        J_{U_{\perp}}^{\Sigma_{1}}
    \end{bmatrix} \vecop(E) = J_{\tilde{q}}^{\Sigma_{1}} \vecop(E),
\end{equation}
where the Jacobian of $U_{\perp}$  with respect to $\Sigma_{1}$ can be derived from \protect{\cite{Vaccaro:1994}} as:
\begin{equation*}
J_{U_{\perp}}^{\Sigma_{1}} = -\Big( U_{\perp}\tr \otimes \big( U_{p} S_{p}^{-1} V_{p}\tr \big) \Big) \, \Pi_{n,p}.
\end{equation*} 
Eventually, the vectorization of $\D\! g(A)[E]$ is
\begin{equation}\label{eq:vec_Lq}
    \vecop(\D\! g(A)[E]) = J_{\exp(A)}^{A} \underbrace{T J_{h}^{\Sigma_{1}} \, \Pi_{n,n} \, J_{\tilde{q}}^{\Sigma_{1}}}_{\eqqcolon J_{A}^{\Sigma_{1}}} \, \vecop(E),
\end{equation}
from which we identify the Jacobian of the exponential with respect to $\Sigma_{1}$, namely,
\begin{equation*}
J_{\exp(A)}^{\Sigma_{1}} = J^{A}_{\exp(A)} J_{A}^{\Sigma_{1}}.
\end{equation*}
Substituting \eqref{eq:vec_Lq_tilde} and \eqref{eq:vec_Lq} into \eqref{eq:vec_LZ_1} and dropping $\vecop(E)$, we obtain the Jacobian of the geodesic with respect to $\Sigma_{1}$
\begin{equation*}
    J_{Z_1}^{\Sigma_{1}} =  \left( \big[ I_{p} \ \ O  \big] \exp(A)\tr  \otimes I_n \right) J_{\tilde{q}}^{\Sigma_{1}} + \left(  \big[ I_{p} \ \ O  \big] \otimes \widetilde{Q} \right) J_{\exp(A)}^{\Sigma_{1}}.
\end{equation*}
By using the same procedure, one can get the Jacobian of the derivative of the geodesic with respect to $ \Sigma_{1} $, i.e.,
\begin{equation*}
\begin{split}
    J_{Z_2}^{\Sigma_{1}}  = & \left( \big[ I_{p} \ \ O  \big] A\tr \exp(A)\tr  \otimes I_n \right) J_{\tilde{q}}^{\Sigma_{1}} +
\left( \big[ I_{p} \ \ O  \big] A\tr \otimes \widetilde{Q} \right) J_{\exp(A)}^{\Sigma_{1}} \\
& + \left(  \big[ I_{p} \ \ O  \big] \otimes \widetilde{Q}\exp(A) \right) J_{A}^{\Sigma_{1}}.
\end{split}
\end{equation*}

\subsection{Jacobians with respect to the tangent vector}

To obtain the Jacobians with respect to the tangent vector $ \Sigma_{2} $, one can proceed similarly to the previous section. The Jacobian of the geodesic with respect to $ \Sigma_{2} $ is given by
\begin{equation*}
    J_{Z_1}^{\Sigma_{2}} = \left( \big[ I_{p} \ \ O  \big] \otimes \widetilde{Q} \right) J_{\exp(A)}^{\Sigma_{2}},
\end{equation*}
and the Jacobian of the derivative of the geodesic with respect to $ \Sigma_{2} $ is
\begin{equation*}
    J_{Z_2}^{\Sigma_{2}} = \left( \big[ I_{p} \ \ O  \big] \otimes \widetilde{Q} \right) \left[ \big( A\tr \otimes I_n \big)  J_{\exp(A)}^{\Sigma_{2}} + \big( I_n \otimes \exp(A) \big)  J_{A}^{\Sigma_{2}} \right].
\end{equation*}
Here,
\begin{equation*}
    J_{\exp(A)}^{\Sigma_{2}} = J_{\exp(A)}^{A} J_{A}^{\Sigma_{2}} \quad \text{and} \quad  J_{A}^{\Sigma_{2}} = T J_{h}^{\Sigma_{2}},
\end{equation*}
with
\begin{equation*}
J_{h}^{\Sigma_{2}} =
\begin{bmatrix}
I_p \otimes \left( \big[ I_{p} \ \ O  \big] \widetilde{Q}\tr \right) \\[4pt]
I_p \otimes \left( \big[ O \ \ I_{n-p} \big] \widetilde{Q}\tr \right) \\[4pt]
               -\Pi_{(n-p),p} \Big( I_p \otimes \big[ O \ \ I_{n-p} \big] \widetilde{Q}\tr \Big) \\[4pt]
               O_{(n-p)^2 \times np}
\end{bmatrix} \in \R^{n^2 \times np}.
\end{equation*}

\section{Condensing} \label{app:condensing_MS}

The linear system \eqref{eq:MS_linear_system} can be solved efficiently thanks to the structure of $J_{F}^{\Sigma}$, which allows any $\delta\Sigma^{(k)}$, $k=2,\ldots,m$, to be expressed as a function of $\delta\Sigma^{(1)}$ \protect{\cite[\S 7.3.5]{Stoer:1991}}.
Eventually, only one linear system of size $2np\times 2np$ has to be solved to find $\delta\Sigma^{(1)}$
\begin{equation*}
   M \cdot \delta\Sigma^{(1)} = - w,
\end{equation*}
where
\begin{equation*}
   M = C + D\cdot\!\!\! \prod_{k=m-1}^{1}\!\! G^{(k)}, \qquad w = F^{(m)} + D\cdot \sum_{k=1}^{m-1} \! \left( \prod_{\ell=k+1}^{m-1} G^{(\ell)} \right) \! \cdot F_{k}.
\end{equation*}
The other $\delta\Sigma^{(k)}$ are obtained as
\begin{equation*}
   \delta\Sigma^{(k)} = F^{(k-1)} + G^{(k-1)} \cdot \delta\Sigma^{(k-1)}, \quad k=2,\ldots, m.
\end{equation*}
The complexity of multiple shooting with this condensing strategy is $O(mn^3p^3)$.

\section*{Data availability}

The code and datasets generated and analyzed during the current study are available in the LFMS repository, \url{https://github.com/MarcoSutti/LFMS_Stiefel}.

%
\section*{Conflict of interest}

The author declares that he has no conflict of interest.

\bibliographystyle{alpha_init}

\bibliography{Shooting_Stiefel.bib}

\end{document}